\documentclass[12pt,a4paper]{article}

\usepackage[utf8x]{inputenc}

\usepackage{amsmath,amssymb,amstext,amsthm}
\usepackage{fancyhdr}
\usepackage{color}
\usepackage[parfill]{parskip}
\usepackage{enumerate}
\usepackage{hyperref}

\newcommand{\N}{\mathbb{N}}
\newcommand{\F}{\mathbb{F}}
\newcommand{\Prim}{\mathbb{P}}
\newcommand{\cha}{\text{$\unlhd$\raisebox{0pt}{$\stackrel{{}_\vert}{}$} }}

\newcommand{\Aff}{\text{\rm Aff($1,p$)}}
\newcommand{\AffF}{\text{\rm Aff($1,F$)}}
\newcommand{\bigslant}[2]{\left.\raisebox{.2em}{$#1$}\middle/\raisebox{-.2em}{$#2$}\right.}
\DeclareMathOperator{\Sym}{\text{\rm Sym}}
\newcommand{\LL}{\mathcal{L}}

\DeclareMathOperator{\RM}{\text{\rm RM}}
\DeclareMathOperator{\Stab}{\text{\rm stab}}
\DeclareMathOperator{\Orb}{\text{\rm orb}}
\DeclareMathOperator{\rk}{\rm rk}

\newcommand{\BIGOP}[1]{\mathop{\mathchoice%
		{\raise-0.22em\hbox{\huge $#1$}}%
		{\raise-0.05em\hbox{\Large $#1$}}{\hbox{\large $#1$}}{#1}}}

\newtheoremstyle{my_th_style1}%
{6pt}
{12pt}
{}
{}
{\bfseries}
{:             }
{.5em}
{}
\newtheoremstyle{my_th_style2}%
{6pt}
{0pt}
{\itshape}
{}
{\bfseries}
{:             }
{.5em}
{}
\newtheoremstyle{my_th_style3}%
{6pt}
{12pt}
{}
{}
{\bfseries}
{:  }
{\newline}
{}

\theoremstyle{my_th_style2}
\newtheorem{Theorem}{Theorem}[section]
\newtheorem{Lemma}[Theorem]{Lemma}	
\newtheorem{Cor}[Theorem]{Corollary}
\newtheorem{Conj}[Theorem]{Conjecture}	
\theoremstyle{my_th_style3}

\theoremstyle{my_th_style1}
\newtheorem{Def}[Theorem]{Definition}     
\newtheorem{Rem}[Theorem]{Remark}	
\newtheorem{Example}[Theorem]{Example}	

\renewcommand*\MakeUppercase[1]{#1}
\renewcommand{\headrulewidth}{0.4pt}
\renewcommand{\footrulewidth}{0.4pt}

\makeatletter
\renewcommand*{\env@matrix}[1][*\c@MaxMatrixCols c]{%
  \hskip -\arraycolsep
  \let\@ifnextchar\new@ifnextchar
  \array{#1}}
\makeatother

\makeatletter 
\renewenvironment{proof}[1][\proofname]{\par 
	\pushQED{\qed}%
	\normalfont \topsep6\p@\@plus6\p@\relax 
	\trivlist 
	\item[\hskip\labelsep 
	\bfseries 
	#1\@addpunct{:}]\ignorespaces 
}{%
\popQED\endtrivlist\@endpefalse 
} 
\makeatother 

\begin{document}
\allowdisplaybreaks
\begin{titlepage}
\vspace*{2cm}
\begin{center}
{\Huge\textbf{Invariant transversals in }}\\
\vspace*{.2cm}
{\Huge\textbf{finite groups}}\\
\vspace*{.6cm}
{\Large von \\ Lucia Christin Ortjohann}\\
\vspace*{1.5cm}
{\Large \textbf{Masterarbeit in Mathematik}}\\
\vspace*{1.5cm}
{\Large vorgelegt der\\
	Fakultät für Mathematik, Informatik und Naturwissenschaften
	der Rheinisch-Westfälischen Technischen Hochschule Aachen}\\
\vspace*{1.5cm}
{\Large im Dezember 2019}\\
\vspace*{2cm}
{\Large angefertigt am Lehrstuhl D für Mathematik}\\
\vspace*{.5cm}
{\Large Erstgutachter: \\ Prof. Dr. Gerhard Hiß}\\
\vspace*{.3cm}
{\Large Zweitgutachterin: \\ Prof. Dr. Alice Niemeyer}\\
\end{center}
\end{titlepage}
\newpage

\fancyhf{}
\renewcommand*\MakeUppercase[1]{#1}
\fancyhead[C]{} 
\fancyhead[R]{\textit{Lucia C. Ortjohann}}
\renewcommand{\headrulewidth}{0.4pt}
\renewcommand{\footrulewidth}{0.4pt}

\section*{Acknowledgement}

I would like to express my gratitude to my supervisor Prof. Dr. Gerhard Hiß for proposing the interesting topic of this thesis and his continued support during its completion. 
Our frequent discussions and his meticulous reading of several drafts substantially helped me improve my thesis.

Furthermore, I would like to thank Prof. Dr. Alice Niemeyer for agreeing to be co-supervisor of this thesis.

Prof. Dr. Gabriele Nebe deserves a special mention for providing the central argument in the proof of Theorem \ref{MainTheoremAbelian}.

My thanks also goes to Darius Andreas Dramburg und Linus Paul Hellebrandt
for always answering my questions, even across time-zones, and providing a fresh view on the comprehensibility of my arguments.

Finally, I would like to thank my family and my friends for their unfailing support and continuous encouragement throughout my years of study. 

\hfill

Lucia C. Ortjohann\\
Aachen, December 2019

\newpage

\tableofcontents
\newpage

\fancyfoot[R]{\textit{Page \thepage}}
\setcounter{page}{1}

\section{Introduction}
\lhead{\slshape 1 \quad Introduction}
Our study of invariant transversals originates in loop theory. 
Loops are non-associative structures which arise, for example, in algebraic applications to geometry. 
Reinhold Baer observed that loops can be studied group-theoretically via the concept of a loop folder which is essentially a triple of a group, a subgroup and a transversal for the set of right cosets.

In research, we find several types of loops, amongst which are right conjugacy closed loops, in short RCC loops. 
Their group-theoretic counterpart, the right conjugacy closed loop folder
is a loop folder such that the transversal is invariant under conjugation by the group.
In this thesis we are especially interested in the existence of invariant transversals, i.e. transversals which are invariant under conjugation by the group, since in general, not every subgroup of a given group has an invariant transversal. 
From the connection mentioned above, we note that studying the existence of invariant transversals is closely linked to studying the existence of RCC loops.
Furthermore, the interplay between loop theory and group theory has prompted discoveries in both areas and led to new questions on both sides.
A summary of the history of algebraic research in loop theory can be found in \cite{articdiss}. 

This thesis is a continuation of the work of Artic and Hiss. 
In \cite{articdiss} Artic establishes basic results on RCC loop folders and computes all non-associative RCC loops of order up to 30 and in \cite{artichiss} the authors Artic and Hiss classify the RCC loops of order $2p$, where $p$ is an odd prime, up to isomorphism.

Our main object of study is an RCC loop folder which consists of a finite group $G$, a subgroup $H$ of $G$ and a transversal $T$ for the set of right cosets of $H$ in $G$ such that $1 \in T$ and $T^g=T$ for all $g \in G$, i.e. $T$ is invariant under conjugation by $G$. 
Artic showed in \cite[(2.6) Remark]{articdiss} that, for a group $G$ and a subgroup $H$ of $G$ with $G' \cap H = \{1\}$, there exists a transversal $T$ for the set of right cosets of $H$ in $G$ such that $(G,H,T)$ is an RCC loop folder. 
We generalize this statement and give conditions under which a converse statement holds.
Based on the classification of RCC loops of order $2p$, we begin classifying RCC loops of order $pq$, where $p$ and $q$ are distinct primes.
We observed in our investigation of the classification and of the computational results from Artic that Frobenius groups frequently appear as well-behaved examples. 
Another noteworthy result from these considerations is
that every abelian group with a small enough subgroup has a transversal for the set of right costs of this subgroup which generates the whole group.
All these observations lead us to the construction of infinite series of RCC loop folders.

This thesis is divided into six chapters. 
In Chapter 2 we give an introduction to loop theory and explain the connection between loops, loop folders and invariant transversals. 
In the third chapter we examine Frobenius groups and discover that, for a Frobenius group $G$ with an abelian Frobenius complement and a subgroup which lies in this complement, there exists an invariant transversal with a special form. 
Furthermore, we give some examples of Frobenius groups which fulfil the above stated conditions. 
In Chapter 4 we start to classify RCC loops of order $pq$, where $p$ and $q$ are distinct primes.
The next chapter concerns the construction of RCC loop folders. First, we construct these loop folders by adding direct and semi-direct factors to an existing RCC loop folder. Then, we construct invariant generating transversal from given RCC loop folders. 
Next, we prove the existence of invariant generating transversals for a subgroup in an abelian group, provided the index of the chosen subgroup is large enough.
In Chapter 6, we state and prove partial converses of \cite[(2.6) Remark]{articdiss} under various assumptions which lead to the consideration of $H$-invariant transversals.

Our group-theoretical notation is standard. 
In particular, we write $G'$ for the commutator subgroup of the group $G$.
A cyclic group of order $n$ is denoted by $C_n$ and a characteristic subgroup $N$ of a group $G$ is denoted by $N \cha G$. The difference of two sets $B$ and $A$ is denoted by $B - A := \{x \in B \mid x \notin A\}$.
We assume all groups to be acting from the right and thus, only consider right transversals.
All groups, loops and quasigroups are finite. 

\newpage
\section[Connection between invariant transversals and RCC loops]{Connection between invariant transversals and RCC loops}

\lhead{\slshape 2 \quad Invariant transversals and RCC loops}

We express our study of invariant transversal in finite groups in the language of loops. Therefore, this chapter provides an introduction to the topic of loops and loop folders.
 First, we give definitions and prove some basic structural observations. 
Then we define right conjugacy closed (RCC) loops and RCC loop folders and obtain some properties of these objects. 
These RCC loop folders are directly connected to invariant transversals.
The whole chapter is based on \cite[Chapter 1 - 2.2]{articdiss}.

\subsection{Loops and loop folders}

To begin with, we introduce basic concepts from the field of loop theory.   

\begin{Def}[Quasigroup and loop]
	A quasigroup $\mathcal{L}$ is a set with a binary operation $\ast : \mathcal{L} \times \mathcal{L} \rightarrow \mathcal{L}$, such that every equation $x \ast a = b$ or $a \ast x = b$ with $a,b \in \mathcal{L}$ has a unique solution $x \in \LL$.
	
	A quasigroup $\mathcal{L}$ is called a loop, if there is an identity element $1_\LL$ of $\mathcal{L}$.
\end{Def}

If the loop is associative, it is a group.

\begin{Def}[Right multiplication group]
	Let $\mathcal{L}$ be a loop and $x \in \mathcal{L}$. We define $R_x$ to be the right multiplication by $x$:
	\begin{align*}
		R_x : \LL \rightarrow \LL , l \mapsto l\ast x.
	\end{align*}
	We set $R_{\LL}:=\{R_x \mid x \in \LL \}$. Then $R_{\LL}$ is a subset of the symmetric group Sym($\LL$). We define the right multiplication group RM($\LL$) of $\LL$ as a subgroup of Sym($\LL$) by
	\begin{align*}
		\RM(\LL):= \langle R_{\LL} \rangle.
	\end{align*}
\end{Def}

\begin{Rem}
	\label{OperationL}
The right multiplication group $\RM(\LL)$ of a loop $\LL$ acts transitively on $\LL$ since for two arbitrary elements $x,y \in \LL$ there exists a uniquely determined element $l \in \LL$ with $x * l = y$ and thus, $xR_l=y$. Furthermore, let $x \in \LL$ such that $l*x=l$ for all $l \in \LL$. As $x$ is uniquely determined, it follows that $x=1_\LL$ and hence, $\RM(\LL)$ acts faithfully on $\LL$.   
\end{Rem}

\pagebreak

\begin{Def}[Envelope of a loop]
	Let $\LL$ be a loop with identity element $1_{\LL}$. The triple $(\RM(\LL),\Stab_{\RM(\LL)}(1_{\LL}),R_{\LL})$ is called the envelop of $\LL$.
\end{Def}

The generalization of envelops of loops are loop folders.

\begin{Def}[Loop folder]
	A triple $(G,H,T)$ is called a loop folder, if $G$ is a finite group, $H$ is a subgroup of $G$ and $T$ is a transversal for the right cosets of every $H^{g}$, $g \in G$, with $1_G \in T$.
	We define the order of $(G,H,T)$ as the size of $T$.
	
	We call the loop folder trivial if $|H|=1$ or $|T|=1$.
\end{Def}

\begin{Def}
	Let $(G,H,T)$ be a loop folder. The group $G$ acts via right multiplication on the set of right cosets $H\backslash G = \{ Ht \mid t \in T \}$ of $H$. This action is transitive. We say that $(G,H,T)$ is faithful if $G$ acts faithfully on $H \backslash G$.
\end{Def}

Now we give an equivalent characterisation of a faithful loop folder. For this purpose, we set
\begin{align*}
\text{\rm core}_G(H):=  \displaystyle\bigcap_{g \in G} H^g.
\end{align*}

\begin{Lemma}
	\label{faithful}
	The loop folder $(G,H,T)$ is faithful if and only if \linebreak
	 $\text{\rm core}_G(H)=\{1_G\}$.
\end{Lemma}

\begin{proof}
	$G$ acts faithfully on $H \backslash G$ if and only if $Htg=Ht$ for all $t \in T$ implies $g=1_G$. Since we have $Htg=Ht$ if and only if $g \in t^{-1}Ht$, it follows that $G$ acts faithfully on $H \backslash G$ if and only if
	\begin{align*}
		\{1_G\}=\bigcap_{t \in T} H^t.
	\end{align*}
	Since for every $g \in G$ there exists $h \in H$ and $t \in T$ such that $g=ht$, it follows that
	\begin{align*}
	\{1_G\}&=\bigcap_{t \in T} H^t =\bigcap_{g \in G } H^g. 
	 &&\qedhere
	\end{align*}
\end{proof}

We characterise the envelop of a loop.

\begin{Lemma}
	\label{Lemma envelop faithful}
	The envelope $(\RM(\LL),\Stab_{\RM(\LL)}(1_{\LL}),R_{\LL})$ of a loop $\LL$ is a faithful loop folder.
\end{Lemma}
\begin{proof}
	We write $(G,H,T)$ for the envelop $(\RM(\LL),\Stab_{\RM(\LL)}(1_{\LL}),R_{\LL})$ of the loop $\LL$. 
	First, we show that $T$ is a transversal for every $H^g$, $g \in G$. Suppose that $R_xR_y^{-1}=g^{-1}hg \in H^g$ for $x,y \in \LL, g \in G$ and $h \in H$ and we set $l:=(1_{\LL})g$. Then we have $lg^{-1}hg=l$ and thus, $lR_xR_y^{-1}=l$. This yields the equation $l*x=l*y$ in $\LL$. Since $\LL$ is a loop the solution of this equation is unique and hence, $x=y$. This gives us directly $R_x=R_y$. Therefore, $T$ contains at most one element of every coset of $H^g$ for all $g \in G$.
	Furthermore, with the orbit stabilzer theorem and the fact that $G$ acts transitively on $\LL$ (see Remark \ref{OperationL}), it follows that
	\begin{align*}
		\frac{|G|}{|H^g|}=\frac{|G|}{|H|}=\frac{|G|}{|\Stab_{\RM(\LL)}(1_\LL)|}=|\Orb(1_\LL)|=|\LL|=|T| \quad \text{for all } g \in G.
	\end{align*}
	This implies that $(G,H,T)$ is a loop folder.
	
	Now we show that $(G,H,T)$ is faithful. 
	We know from Remark \ref{OperationL} that $G$ acts transitively and faithfully on $\LL$. Thus, the operation of $G$ on the right cosets of $H=\Stab_{G}(1_\LL)$ in $G$ is equivalent to the action of $G$ on $\LL$ and we can conclude that $G$ acts faithfully on $H \backslash G$. 
\end{proof}

\begin{Rem}
	Let $\LL$ be a loop and $(G,H,T)$ be the envelop of $\LL$. Suppose that $G$ is abelian. Then $H=\text{\rm Core}_G(H)=\{1_G\}$ because the envelop is a faithful loop folder. Thus, $G=T \cong \LL$ and we can conclude that $\LL$ is associative. Hence, $\LL$ is a group. Furthermore, an associative loop is a group. Thus, the right multiplication group of a non-associative loop is non-abelian.
\end{Rem}

Given a loop folder $(G,H,T)$ we can construct a loop $(T,\ast)$ on the set $T$.  
However, the envelop of $T$ need not be equal to $(G,H,T)$. Only a faithful loop folder $(G,H,T)$ with $G=\langle T \rangle $ is an envelope of a loop.

\begin{Def}
	Let $G$ be a group and $H$ be a subgroup of $G$. A transversal $T$ for $H$ in $G$ is called a generating transversal if $G=\langle T \rangle$.
\end{Def}

\begin{Lemma}
	\label{LoopFromLoopFolder}
	Let $(G,H,T)$ be a loop folder. Define a multiplication $*$ on $T$ by $t_1 * t_2 = t_3$, where $t_3$ is the uniquely determined element in $Ht_1t_2 \cap T$. Then $(T,*)$ is a loop.
\end{Lemma}

\begin{proof}
	Clearly, this multiplication is a binary operation since the solution is a unique element. Let $s,t \in T$. Then we show that the equations $x * t =s$ and $t*x=s$ have a unique solution $x \in T$. Suppose that the equation $x*t=s$ has two solutions $x_1,x_2 \in T$. Then it follows that $Hx_1t=Hx_2t$, thus, $Hx_1=Hx_2$. 
	Since $T$ is a transversal for $H \backslash G$ , we have $x_1=x_2$. Again, suppose that the equation $t*x=s$ has two solutions $x_1,x_2 \in T$. 
	Then this yields that $Htx_1=Htx_2$, hence, $H^tx_1=H^tx_2$. 
	Because $T$ is a transversal for $H^t \backslash G$, this implies that $x_1=x_2$. Furthermore, we have $1_G \in T$ and $1_G * t = t = t * 1_G$. Therefore, $T$ is a loop.
\end{proof}

In order to proof the next theorem we need an auxiliary statement.

\begin{Lemma}
	\label{HelpForThis}
	Let $(G,H,T)$ be a loop folder and let $(T,*)$ be the loop defined in Lemma \ref{LoopFromLoopFolder}. Then we have
	\begin{align*}
		H((x_1 *x_2)\dotsb*x_n)=Hx_1 \dotsb x_n \quad\text{for all } x_1, \dotsc, x_n \in T.
	\end{align*}
 Hence, the element $((x_1 *x_2)\dotsb*x_n)$ is the uniquely determined element of $Hx_1 \dotsb x_n \cap T$.
\end{Lemma}

\begin{proof}
	We show this statement by induction over $n$. By definition the base case is clear. Thus, suppose that  $H((x_1 *x_2)\dotsb*x_{n-1})=Hx_1 \dotsb x_{n-1} $ for $n-1 \geq 1$ and set $t:=((x_1 *x_2)\dotsb*x_{n-1})$. Then it follows from the definition and the induction hypothesis that 
	\begin{align*}
		H((x_1 *x_2)\dotsb*x_{n})&=H(t * x_n)=Htx_n=Hx_1\dotsb x_{n-1}x_n.
		&&\qedhere
	\end{align*}
\end{proof}

Now we show that a faithful loop folder $(G,H,T)$ with $G=\langle T \rangle $ is an envelope of a loop.

\begin{Theorem}
	\label{envelop}
	Let $(G,H,T)$ be a faithful loop folder and suppose that $T$ is a generating transversal. Then the envelop of the loop $(T,*)$ defined in Lemma \ref{LoopFromLoopFolder} is isomorphic to $(G,H,T)$. Hence, a faithful loop folder with a generating transversal is an envelop of a loop.
\end{Theorem}
\begin{proof}
	We denote the envelop of $(T,*)$ by $(\widetilde{G},\widetilde{H},\widetilde{T})$. Now we show that 
	\begin{align*}
		\Phi: \widetilde{G} \rightarrow G, R_{x_1} \dotsb R_{x_n} \mapsto x_1 \dotsb x_n
	\end{align*}
	is a group isomorphism such that $\Phi(\widetilde{H})=H$ and $\Phi(\widetilde{T})=T$.
	
	First, we show that $\Phi$ is well defined. Let $ R_{x_1} \dotsb R_{x_n} = R_{y_1} \dotsb R_{y_n} \in \widetilde{G}$. Then 
	\begin{align*}
	(t * x_1 )\dotsb * x_n = (t * y_1)\dotsb * y_n \quad \text{for all } t\in T.
	\end{align*}
 	Now Lemma \ref{HelpForThis} yields that 
	\begin{align*}
	Htx_1 \dotsb x_n = Hty_1 \dotsb y_n \quad \text{for all } t\in T.
	\end{align*}
	Since $(G,H,T)$ is a faithful loop folder, it follows from 
	\begin{align*}
	H^tx_1 \dotsb x_n = H^ty_1 \dotsb y_n \quad \text{for all } t\in T
	\end{align*}
	that $x_1 \dotsb x_n = y_1 \dotsb y_n$.
	
	Clearly, $\Phi$ is a group homomorphism.
	Furthermore, it follows directly that $\Phi(\widetilde{T})=\Phi(R_T)=T$ and hence, we have 
	\begin{align*}	
	\Phi(\widetilde{G})=\Phi(\langle R_T \rangle)= \langle \Phi(R_T) \rangle =\langle T \rangle=G.
	\end{align*}
	Assume that $x_1 \dotsb x_n = y_1 \dotsb y_n \in G$. This implies that
	\begin{align*}
	Htx_1 \dotsc x_n = Hty_1 \dotsc y_n \quad \text{for all } t\in T.
	\end{align*} and hence, with Lemma \ref{HelpForThis} it follows that
	\begin{align*}
	(t * x_1 )\dotsb * x_n = (t * y_1)\dotsb * y_n \quad \text{for all } t\in T.
	\end{align*}
	This gives us  $ R_{x_1} \dotsb R_{x_n} = R_{y_1} \dotsb R_{y_n}$ and thus, $\Phi$ is bijective.
	
	Finally, we show that $\Phi(\widetilde{H})=H$. Let $h \in \widetilde{H}$. Then we have $h=R_{x_1} \dotsb R_{x_n}$ for some $x_1, \dotsc, x_n \in T$ and it follows that 
	\begin{align*}
	1_{(T,*)}=1_{(T,*)}  h=1_{(T,*)}(R_{x_1} \dotsb R_{x_n})=(x_1 * x_2)\dotsb * x_n
	\end{align*}
	Note that $1_G=1_{(T,*)}$.
	Again, Lemma \ref{HelpForThis} implies that
	\begin{align*}
	H=H1_{(T,*)}=H((x_1 * x_2)\dotsb * x_n)=Hx_1\dotsb x_n
	\end{align*}
	and hence, $x_1 \dotsb x_n \in H$. We obtain $\Phi(h) \in H$ and thus, $\Phi(\widetilde{H}) \subseteq H$. Moreover, this yields that 
	\begin{align*}
	 \big|\widetilde{T}\big|=\frac{\big|\widetilde{G}\big|}{\big|\widetilde{H}\big|}=\frac{\big|G\big|}{\big|\Phi(\widetilde{H})\big|} \geq \frac{\left|G\right|}{\left|H\right|} = \left|T\right|.
	\end{align*}
	As $|\widetilde{T}|=|T|$, we have equality and hence, we have $\Phi(\widetilde{H})=H$. 
\end{proof}

\subsection{RCC loops and RCC loop folders}

RCC loops and RCC loop folders form the connection between loops and invariant transversals. Therefore, we introduce this specialization of loops and loop folders.

\begin{Def}[RCC loop]
	A loop $\LL$ is called right conjugacy closed (RCC) if the set $R_{\LL}$ is closed under conjugation. I.e. for all $x,y \in \LL$ we have
	$R_x^{-1}R_y R_x \in R_{\LL}$.
\end{Def}

\begin{Def}[RCC loop folder]
	A loop folder $(G,H,T)$ is called right conjugacy closed (RCC) if the transversal $T$ is $G$-invariant under conjugation, i.e. $g^{-1}tg \in T$ for all $g \in G, t \in T$. 
\end{Def}

The property of being right conjugacy closed is preserved by our previous constructions as we see in the next lemma.

\begin{Lemma}
	Let $(G,H,T)$ be an RCC loop folder. Then the loop defined on $T$ in Lemma \ref{LoopFromLoopFolder} is an RCC loop. Further, if $\LL$ is an RCC loop, then the envelop of $\LL$ is an RCC loop folder.
\end{Lemma}
\begin{proof}
	In Lemma \ref{LoopFromLoopFolder} we showed that $(T,*)$ is a loop. Hence, we only have to show that the set $R_{T}$ is closed under conjugation, i.e. for all $s,t \in T$ we have
	$R_s^{-1}R_t R_s \in R_{T}$. Let $s,t,y \in T$. Since $T^g=T$ for all $g \in G$, we have $x:=s^{-1}ts \in T$. It follows from $Hyts=Hysx$ and Lemma \ref{HelpForThis} that
	\begin{align*}
		y(R_tR_s)=(y*t)*s=(y*s)*x=y(R_sR_x)
	\end{align*}
	As $y$ is arbitrary, we conclude $R_s^{-1}R_t R_s=R_x \in R_T$. The second statement follows directly from the definition.
\end{proof}

If $G$ is a group, $H \leq G$ and $G' \cap H = \{1\}$, then there exists a $G$-invariant transversal for $H \backslash G$. 

\begin{Lemma}
	\label{GStrichSchnitH}
	Let $G$ be a group and let $H \leq G$ be a subgroup of $G$ with $H \cap G' = \{1\}$. Furthermore, let $S \subseteq G$ with $1 \in S$ be a transversal for $HG'$ in $G$. Set $T:=G'S$. Then $(G,H,T)$ is an RCC loop folder.
\end{Lemma}

\begin{proof}
	Let $g \in G$. Then there exists $s \in S, x \in G'$ and $h \in H$ such that $g=(hx)s=h(xs) \in Hxs$ and $xs \in T$. Moreover, we have 
	\begin{align*}
	\left| T\right| =  \left| G' \right|\left| S \right|= \left| G' \right|\left| G:HG' \right| =\left| G : H \right| 
	\end{align*} and thus, $T$ is a transversal for $H \backslash G$.
	Now let $t = xs \in T$ with $x \in G'$ and $s \in S$. Then we have 
	\begin{align*}
		g^{-1}tg = g^{-1}tgt^{-1}t= ([g,t^{-1}]x)s \in G's \subseteq T
	\end{align*}
	for all $g \in G$. Hence, $T$ is $G$-invariant and the statement follows.
\end{proof}

There might be $G$-invariant transversals which are not of the form $G'S$, where $S$ is a right transversal of $HG'$ in $G$, but we see in Chapter 3 that if $G$ is a Frobenius group with an abelian Frobenius complement, every $G$-invariant transversal is of the form $G'S$.

Finally, we give a nice property of RCC loop folders, which we will use in Chapter 6.

\begin{Lemma}
	\label{NormalizerCentralizer}
	 If $(G,H,T)$ is an RCC loop folder, then 
	 $N_G(H)=H C_G(H)$.
\end{Lemma}

\begin{proof}
	Clearly, we have $H C_G(H) \leq N_G(H)$. Conversely, let $g \in N_G(H)$. Then there exist $t \in T$ and $h \in H$ such that $g=ht$. As $h \in N_G(H)$, we obtain $t \in N_G(H)$. Since $T$ is $G$-invariant, we have $h'th'^{-1} \in T$ for all $h' \in H$ and it follows with $h'th'^{-1}t^{-1} \in H$ that $t=h'th'^{-1}$ for all $h' \in H$. Thus, $t \in C_G(H)$ and we can conclude that $g=ht \in HC_G(H)$.
\end{proof}

\newpage

\section{Frobenius groups}
\lhead{\slshape 3 \quad Frobenius groups}

First, we look at Frobenius groups in general. We see that for a Frobenius group $G$ and a subgroup $H \leq G$ exists a $G$-invariant transversal for $H\backslash G$ with a special form if the Frobenius complement is abelian and $H$ lies in this complement. 
Then we examine the affine groups of dimension 1 and show that theses groups are Frobenius groups with abelian complements and finally, we give another generic example of those Frobenius groups which are not affine groups.

\subsection{The general case}

Here we define Frobenius groups and give some properties of these groups in the next theorem and next lemma.

\begin{Def}[{\cite[Definition (7.1)]{isaacs}}]
	Let $C \leq G$, with $1 \lneq C \lneq G$. Assume that $C \cap C^g= \{1\}$ whenever $g \in G-C$. Then $C$ is a Frobenius complement in $G$. A group which contains a Frobenius complement is called a Frobenius group.
\end{Def}

\begin{Theorem}[{\cite[Theorem (7.2) and Lemma (7.3)]{isaacs}}]
	\label{HauptsatzFrobenius}
	Let $G$ be a Frobenius group with Frobenius complement $C$ and put
	\begin{align*}
	 N:=G - \bigcup_{g \in G} (C - \{1\}) ^g.
	\end{align*}
	 Then $N$ is a normal subgroup of $G$ and $G$ is a semidirect product of $N$ and $C$. The normal subgroup $N$ is uniquely determined by $C$ and is called the Frobenius kernel of $G$.
\end{Theorem}

\begin{Lemma}[{\cite[Problem (7.1)]{isaacs}}]
	\label{Isaacs}
 Let $N \trianglelefteq G$ and $ \{1 \} \lneq C \lneq G$ with $G=NC$ and $N \cap C = \{1\}$. Then the following statements are equivalent: 
\begin{enumerate}[a)]
 	\item $C_G(n) \leq N$ for all $1 \neq n \in N$;
 	\item $C_C(n)=\{1\}$ for all $1 \neq n \in N$;
 	\item $C_G(c) \leq C$ for all $1 \neq c \in C$;
 	\item Every $x \in G - N$ is conjugate to an element of $C$;
 	\item If $1 \neq c \in C$, then $c$ is conjugate to every element of $Nc$.
 	\item $C$ is a Frobenius complement in $G$.
 \end{enumerate}
\end{Lemma}
\begin{proof}
\textbf{a) $\Rightarrow $ b):} Let $1\neq n \in N$ and let $c \in C_C(n)$. Then $cn=nc$ and thus, $c \in C_G(n) \leq N$. It follows that $c \in N \cap C = \{1\}$ and hence, statement b) follows.

\textbf{b) $\Rightarrow $ c):}
Let $1\neq c \in C$ and let $g \in C_G(c)$. We have $g=nh$ for some $n \in N, h \in C$ and we have $c=c^g=c^{nh}$. It follows that $n^{-1}cn=hch^{-1} \in C$. Thus, $c^{-1}n^{-1}cn \in C$ and since $N \trianglelefteq G$, we have $c^{-1}n^{-1}cn \in N$. 
This implies that $cn=nc$ and it follows that $c \in C_C(n)$. As  $c \neq 1$, we have $n=1$ by b), and we conclude $g=h \in C$.

\textbf{c) $\Rightarrow $ d):}
Let $1 \neq c \in C$. We denote the conjugacy class of $G$ containing $c$ by $C^G_c$ and the conjugacy class of $C$ containing $c$ by $C^C_c$. By our assumption we have $C_G(c)=C_C(c)$. Thus, we have
\begin{align*}
\left|C_c^G\right|=\left|G:C_G(c)\right|=\left|N\right|\left|C:C_C(c)\right|=\left|N\right|\left|C_c^C\right|.
\end{align*}
Moreover, we have $C_c^G \subseteq G - N$, since  $c \notin N $ and $N \trianglelefteq G$.

Let $1 \neq c_1,c_2 \in C$ and suppose that $g^{-1}c_1g=c_2$ for some $g \in G$. Then $g=nc'$ for some $n \in N$ and $c' \in C$. It follows that $c'^{-1}n^{-1}c_1nc'=c_2$ and this implies with $N \trianglelefteq G$ that
\begin{align*}
	c_1^{-1}n^{-1}c_1n=c_1^{-1}c'c_2c'^{-1} \in N \cap C = \{1\}
\end{align*}
and hence, we have $c'^{-1}c_1c'=c_2$. 
We conclude that two elements of $C - \{1\}$ are conjugated in $G$ if and only if there are conjugated in $C$.

Now it follows that 
\begin{align*}
 \left|\bigcup_{1 \neq c \in C} C_c^G\right|=\left|N\right|\left|\bigcup_{1 \neq c \in C} C_c^C\right|=\left|N \right| (\left|C -1  \right|) = \left|G \right|- \left|N \right|.
\end{align*}
We conclude that $G - N = \bigcup_{1 \neq c \in C} C_c^G $ and thus, if $x \in G-N $, then $x \in C_c^G$ for some $1 \neq c \in C$.

\textbf{d) $\Rightarrow $ e):}
Let $1 \neq c \in C$ and let $n \in N$. Then we have $nc \in G- N$ and thus, $nc$ is conjugate to some $h \in C$ by d). This implies that there exists $g=mc' \in G$ with $m \in N, c' \in C$ such that $(nc)^g=h$. Since $N \trianglelefteq G$, we have
\begin{align*}
	(Nc)^g&=g^{-1}Ncg=c'^{-1}m^{-1}Ncmc'=c'^{-1}Ncmc^{-1}cc'\\
	&=c'^{-1}Ncc'=Nc'^{-1}cc'=Nc^{c'}.
\end{align*}
 Hence, there exists $n' \in N$ such that $h=(nc)^g=n'c^{c'}$. This implies $n'=h(c^{c'})^{-1} \in N \cap C = \{1\}$. Therefore, $h$ is conjugate to $c$ and thus, $c$ is conjugate to $nc$, which proves the statement, since $n$ was chosen arbitrarily.

\textbf{e) $\Rightarrow $ f):}
Let $g \in G - N$ with $g=cn$ for some $1 \neq c \in C$ and $n \in N$. By our assumption $c$ is conjugate to $g=cn$. Since $g$ was chosen arbitrarily, we conclude
\begin{align*}
	G - N \subseteq \bigcup_{g \in G} (C^g - \{1\}) = \bigcup_{i=1}^t (C_i - \{1\})
\end{align*}
with $\{C^g \mid g \in G \} = \{ C_1, \dotsc , C_t \}$ and $t=|G:N_G(C)|$. It follows that
\begin{align*}
\left|G - N\right| &\leq \displaystyle\sum_{i=1}^{t}(\left|C_i\right|-1)=t(\left|C\right|-1) \leq \left|G:C\right|(\left|C\right|-1)\\
&= \left|N\right|(\left|C\right|-1)=\left|C\right|\left|N\right|-\left|N\right|
=\left|G - N\right|.
\end{align*}
Thus, we have equality and this yields that $C_i \cap C_j = \{1\}$ for all $1 \leq i \neq j \leq t$. Hence, $C$ is a Frobenius complement in $G$.

\textbf{f) $\Rightarrow $ a):}
First, we prove that $C \cap C_G(n) =\{1\}$ for every $1 \neq n \in N$. Let  $1 \neq n \in N$ and let $h \in C \cap C_G(n)$. Then we have $h=h^n \in C^n$. As $1 \neq n \in G - C$, it follows that $h \in C \cap C^n = \{1\}$. Hence, the statement follows.

Let $1\neq x \in N$ and let $y \in C_G(x)$. Assume that $y \notin N$. Since $G$ is a Frobenius group, it follows from Theorem \ref{HauptsatzFrobenius} that $N=G - \displaystyle\bigcup_{g \in G} (C - \{1\}) ^g$. Thus, there exists $z \in G$ such that $y \in C^z$. Let $c \in C$ such that $y=c^z$.
It follows that
\begin{align*}
x=y^{-1}xy=(c^z)^{-1}xc^z=z^{-1}c^{-1}x^{z^{-1}}cz.
\end{align*}
Hence, $x^{z^{-1}}=c^{-1}x^{z^{-1}}c$ and thus, $c \in C_G(x^{z^{-1}})\cap C$. As $1 \neq x^{z^{-1}} \in N$, it follows from the above statement that $c=1$. But then we have $y=c^z=1$ and this is a contradiction to $y \notin N$. Therefore, we have  $y \in N$.
\end{proof}

Now we examine the properties of Frobenius groups with abelian Frobenius complement in detail.

\begin{Lemma}
	\label{abliancomplement}
	Let $G$ be a Frobenius group with Frobenius complement $C$ and Frobenius kernel $N$. Then $C$ is abelian if and only if $N=G'$.
\end{Lemma}
\begin{proof}
	First, let $C$ be abelian. Then $G/N \cong C$ is abelian and thus, $G' \leq N$. Now let $n \in N$ and $1 \neq c \in C$. Then with Lemma \ref{Isaacs} e), it follows that there exists $g \in G$ with $nc=g^{-1}cg$. Hence, $n=[g,c^{-1}] \in G'$ and we conclude that $N=G'$.\\
	Conversely, if $N=G'$, we have that $C \cong G / N =G / G'$ is abelian. 
\end{proof}

\begin{Lemma}
	\label{FrobeniusConjugacyClass}
		Let $G=C \ltimes N$ be a Frobenius group with abelian Frobenius complement $C$ and Frobenius kernel $N$. 
		Then the conjugacy class of $G$ containing $\tau$ is equal to $N\tau$ for all $\tau \in G - N$.
\end{Lemma}

\begin{proof}
	Note that by the assumption that $C$ is abelian, we have $N=G'$ with Lemma \ref{abliancomplement}.
	Let $\tau \in G - N$ and let $C_{\tau}$ denote the conjugacy class of $G$ containing $\tau$. We show that $C_{\tau}=G' \tau $.
	Indeed, we have $C_\tau \subseteq G' \tau $, since $g^{-1}\tau g= [g, \tau^{-1}]\tau \in G' \tau$ for all $g \in G$. 
	Moreover, it follows from Lemma \ref{Isaacs} that $G'\tau=N \tau $ is a conjugacy class of $G$.
	Hence, $C_\tau = G' \tau$.
\end{proof}

 The $G$-invariant transversals of $H\backslash G$ containing the identity element, have a special form if we impose the following conditions: $G$ is a Frobenius group with an abelian Frobenius complement and $H$ is a subgroup of the Frobenius complement. This is proved in the next theorem. Furthermore, we give a construction for all these transversals and we count the number of pairwise distinct $G$-invariant transversals of $H \backslash G$ containing 1.

\begin{Theorem}
	\label{Frobeniustransversale}
	Let $G=C \ltimes N$ be a Frobenius group with abelian Frobenius complement $C$ and Frobenius kernel $N$. 
	Moreover, let $H \leq C$ and let $T$ be a $G$-invariant transversal for $H \backslash G$ with $1 \in T$, i.e. $(G,H,T)$ is an RCC loop folder.
	 Then we have 
	\begin{align*}
		T=N \;\dot\cup \; N\tau_1 \; \dot\cup \;\dotsb \; \dot\cup \;N\tau_n
	\end{align*} 
 with $n=|C:H|-1$ and $\tau_i \in G-N$ for all $1 \leq i \leq n$. In particular, $\{1,\tau_1, \dotsc, \tau_n \}$ is a transversal for the set of right cosets of $HN$ in $G$.
\end{Theorem}

\begin{proof}
	Let $C_{\tau}$ denote the conjugacy class of $G$ containing $\tau \in G$.
	Note that by the assumption that $C$ is abelian, we deduce $N=G'$ with Lemma \ref{abliancomplement} and Lemma \ref{FrobeniusConjugacyClass} yields that $C_\tau=N\tau$ for all $\tau \in G - N$.
	
	Furthermore, we have $|T|=|G:H|=|N||C:H|$ and since $T - N$ is a union of conjugacy classes of the form $N\tau$, we have that $|N|$ divides $|T- N |$. Therefore, the order of $N$ divides $|T|-|T - N|$. As $|T|-|T - N| > 1$, we have $|T|-|T - N| \geq |N|$ and we deduce $N \subseteq T$.
	In conclusion, we have  $T=N \;\dot\cup \; N\tau_1 \; \dot\cup \;\dotsb \; \dot\cup \; N\tau_n $ for some $\tau_i \in G - N$ and $1 \leq i \leq n$.
	
	Moreover, for $S:=\{1,\tau_1 ,\dotsc, \tau_n \}$ we have 
	\begin{align*}
	G=\displaystyle\bigcup_{t\in T}Ht=\displaystyle\bigcup_{n\in N, \tau \in S}Hn\tau=\bigcup_{\tau \in S}HN\tau
\end{align*}
and we have $|S|=|C:H|=|G:HN|$. Thus, $S$ is a transversal for the set of right cosets of  $HN$ in $ G$.
\end{proof}

\begin{Rem}
	\label{BerechnungForbeniusRCCtransversale}
	Let the assumptions and notation be as in Theorem \ref{Frobeniustransversale} and let $S$ be a transversal of $HN$ in $G$ with $1 \in S$. Then it follows from Lemma \ref{GStrichSchnitH} that $G'S=NS$ is a $G$-invariant transversal for $H \backslash G$ with $1 \in S$, since $H \cap G' = \{1\}$. 
	
	From Theorem \ref{Frobeniustransversale}, we know that every $G$-invariant transversal $T$ of $H \backslash G$ with $1 \in T$ is of the form $NS$, where $S$ is a transversal for the set of right cosets of $HN$ in $G$ with $1 \in S$. 
	
	Thus, in order to compute every RCC loop folder with $G$ and $H$ fixed, i.e.\ every $G$-invariant transversal of $H \backslash G$ containing 1, we only have to extend every transversal $S$ of $HN$ in $G$ with $1 \in S $ 
	to the $G$-invariant transversal $NS$ of $H \backslash G$. 
 	\end{Rem}

\begin{Cor}
	Let the assumptions and notation be as in Theorem \ref{Frobeniustransversale}. Then there are $|H|^{|C:H|-1}$ pairwise
	distinct $G$-invariant transversals for $H$ in $G$ containing 1.
\end{Cor}

\begin{proof}
	From Remark \ref{BerechnungForbeniusRCCtransversale} we know that the set
	\begin{align*}
		\mathcal{T}:=\{ NT \mid T \text{ tranversal for } HN \backslash G, 1 \in T \}
	\end{align*}
	contains every $G$-invariant transversal for $H \backslash G$ containing 1 and that every transversal of this set is a $G$-invariant transversal for $H \backslash G$. Thus, we want to show that $|\mathcal{T}|=|H|^{|C:H|-1}$.
	
	Let $S:=\{1,s_1,\dotsc,s_r\}$ be a transversal for $HN\backslash G$. Then we have 
	\begin{align*}
		r=|G:HN|-1=|C:H|-1.
	\end{align*}
	We set 
		$\mathcal{S}:=\{ \{1,h_1s_1, \dotsc, h_rs_r\}  \mid h_i \in H, 1 \leq i \leq r  \}$ and hence, it follows that
	\begin{align*}
				\mathcal{T}&=\{ NT \mid T \text{ tranversal for } HN \backslash G, 1 \in T \}\\
				&=\{ NT \mid T=\{1,n_1h_1s_1, \dotsc, n_rh_rs_r\}  \text{ with } n_i \in N, h_i \in H, 1 \leq i \leq r  \}\\
				&=\{ NT \mid T \in \mathcal{S}\}.
	\end{align*}

	Let $T_1,T_2 \in \mathcal{S}$ with $NT_1=NT_2$ and let $t_1 \in T_1$. Then we have $t_1=nt_2$ for some $n \in N$ and some $t_2 \in T_2$.
	Furthermore, we have $t_1=hs_i$ and $t_2=h's_j$ for some $1 \leq i,j \leq r$ and for some $h,h' \in H$. 
	The fact $h s_i = n h' s_j$ implies that
	\begin{align*}
		s_i s_j^{-1}=h^{-1}nh'=h^{-1}h'h'^{-1}nh' \in HN.
	\end{align*}
	Since $S$ is a transversal for $HN \backslash G$, it follows that $i=j$. This implies that $1=h^{-1}nh'$ and thus, $n=hh'^{-1} \in H \cap N = \{1\}$. We conclude that $t_1=t_2$ and hence, we have $T_1=T_2$.
	
	Now it follows that $|\mathcal{T}|=|\mathcal{S}|=|H|^r=|H|^{|C:H|-1}$.
\end{proof}

\begin{Rem}
	\label{BerechnungFrobeniusGruppenRCC}
	Let the assumptions and notation be as in Theorem \ref{Frobeniustransversale} and suppose additionally that $G=\langle T \rangle$. Since $G$ is a Frobenius group and we have $H \leq C$, it follows that $\text{\rm core}_G(H)=\{1\}$. Thus, $(G,H,T)$ is a faithful RCC loop folder with $G=\langle T \rangle$ and it follows with Theorem \ref{envelop}  that $(G,H,T)$ is an envelope of an RCC loop. 
	Furthermore, it follows from Theorem \ref{Frobeniustransversale} that $G=\langle T \rangle = \langle N, \tau_1 , \dotsc, \tau_n \rangle$ and hence, $C \cong G/N=\langle N\tau_1, \dotsc, N\tau_n \rangle$. To construct every envelope of an RCC loop with right multiplication group $G$ and stabilizer $H$, we only have to extend every transversal $S$ of $HN$ in $G$ with $1 \in S$, to a $G$-invariant transversal $T:=NS$ of $H \backslash G$ and test whether we have $ C \cong \langle N\tau_1, \dotsc, N\tau_n \rangle $.
\end{Rem}

\subsection{Affine groups}
In this section we give a generic example of Frobenius groups. We see that every 1-dimensional affine group over a field  is a Frobenius group and that we can apply the theory of the previous section to these groups. First we give a definition of the 1-dimensional affine group over a field.

\begin{Def}[{\cite[Definition p. 52]{dixon}}]
	Let $F$ be a field.
  The set of all permutations of $F$ of the form 
  \begin{align*}
	  t_{\alpha,\beta}: F \rightarrow F, x \mapsto \alpha x + \beta, \quad \alpha, \beta \in F \text{ and } \alpha \neq 0
  \end{align*} constitutes a subgroup of $\Sym (F) $ with composition as group multiplication.
   This group is called the 1-dimensional affine group over $F$ and is denoted by Aff($1,F$). In the special case, where $F$ is a finite field of order $q$, we denote the group by Aff$(1,q)$.
     
   For multiplying two elements $t_{\alpha,\beta}, t_{\gamma,\delta} \in \AffF$ we can use the following formula:
	\begin{align*}
(t_{\alpha,\beta} \circ t_{\gamma,\delta})(x)=t_{\alpha,\beta}(\gamma x + \delta ) = \alpha \gamma x + \alpha \delta + \beta = (t_{\alpha \gamma, \alpha \delta + \beta})(x)  \text{ for all } x \in F.
\end{align*}
	It follows directly that 
	\begin{align*}
		(t_{\alpha,\beta})^{-1}=t_{\alpha^{-1},-\alpha^{-1}\beta}
	\end{align*}
	
	and thus, conjugation is described by this formula:
	\begin{align}
	\label{conjugationformel}
	(t_{\alpha,\beta})^{-1} \circ t_{\gamma, \delta} \circ t_{\alpha,\beta}=t_{\gamma, \alpha^{-1}\beta\gamma+\alpha^{-1}\delta-\alpha^{-1}\beta}.
	\end{align}
\end{Def}

The following remark shows that the affine groups are Frobenius groups with abelian complement and the next theorem gives us useful properties of the affine groups.

\begin{Rem}
	\label{AffineRemark}
	Let $F$ be a field and put $P:=\{t_{1,\beta} \mid \beta \in F \} \leq \AffF$ and $L:=\{t_{\alpha,0} \mid \alpha \in F^* \} \leq \AffF$. Then we have $P \cong (F,+)$ and $ L \cong F^*$. It is easy to see that $P$ is a normal abelian subgroup of Aff$(1,F)$ and $L$ is an abelian subgroup. Furthermore, we have $P \cap L = \{t_{1,0}\}=\{\text{Id}\}$ and Aff($1,F)=LP$. Thus, Aff$(1,F) = L \ltimes P$. 
	
	Let $t_{\beta,\gamma} \in G - L$. 
	Thus, $\beta,\gamma \neq 0 $ and with the equation \eqref{conjugationformel} we have $L^{t_{\beta,\gamma}}=\{ t_{\alpha,\alpha \beta^{-1} \gamma - \beta^{-1} \gamma } \mid \alpha \in F^* \}$. 
	Since $\alpha \beta^{-1} \gamma - \beta^{-1} \gamma \neq 0$ for all $1 \neq \alpha \in F^*$, we have $L \cap L^{t_{\beta,\gamma}}=\{t_{1,0}\}$. Hence, $L$ is an abelian Frobenius complement of Aff($1,F$) and Aff($1,F$) is a Frobenius group.
\end{Rem}

\begin{Theorem}[{\cite[II, Satz 3.6]{huppert}}]
	\label{HuppertSatz3.6}
	Let $G$ be a transitive permutation group of prime degree $p$. Then the following statements are equivalent:
	\begin{enumerate}[a)]
		\item $G$ is solvable.
		\item $G$ has a normal Sylow $p$-subgroup.
		\item $G$ is permutation isomorphic to a subgroup of the affine group $\Aff$.
		\item $G$ is a Frobenius group.
	\end{enumerate}
\end{Theorem}

In conclusion, we can apply Theorem \ref{Frobeniustransversale}, Remark \ref{BerechnungForbeniusRCCtransversale} and Remark \ref{BerechnungFrobeniusGruppenRCC} to every subgroup of $\Aff$ properly containing $P$ as defined in Remark \ref{AffineRemark}.

In Chapter \ref{Infinite series of right multiplication groups} we see that the affine groups play also an important role as right multiplication groups of RCC loops of order $pq$.

\subsection{Further Frobenius groups}
Here we give another generic example of a Frobenius group with an abelian complement. These groups occur naturally as a point stabilizer in a Suzuki group. Suzuki groups are Zassenhaus groups and any subgroup of a Zassenhaus group fixing a point is a Frobenius group \cite[Chapter XI]{huppert3}.

\begin{Lemma}[{\cite[3.1 Lemma]{huppert3}}]
Let $K=\F_q$ with $q=2^{2m+1}$ and $m>0$. Then $K$ has exactly one automorphism $\pi$ such that $\pi^2(x)=x^2$ for $x \in K$, namely $\pi(x)=x^{2^{m+1}}$. For $a,b \in K$ and $\lambda \in K^*$, let 
\begin{align*}
	S(a,b) = 
	\begin{pmatrix}
		1 & 0 & 0 & 0 \\
		a & 1 & 0 & 0 \\
		b & \pi(a)& 1 &0 \\
		\pi(a)a^2+ab + \pi(b) & \pi(a)a+b &a &1 
	\end{pmatrix}
\end{align*}
 and
 \begin{align*}
	 M(\lambda)=
	 \begin{pmatrix}
		 \lambda^{1+2^m} & 0& 0 & 0 \\
		 0 & \lambda^{2^m} & 0 &0 \\
		 0 & 0 & \lambda^{-2^m}& 0 \\
		 0 &0 &0 &\lambda^{-1-2^m}
	 \end{pmatrix}.
 \end{align*}
 
The set 
 	\begin{align*}
 		\mathfrak{F}=\{ S(a,b) \mid a,b \in K\}
 	\end{align*}
 	is a group of order $q^2$ and the set 
 	\begin{align*}
	 	\mathfrak{H} = \{ M(\lambda) \mid \lambda \in K^* \}
 	\end{align*} 
	is a group isomorphic to $K^*$.
	Furthermore, $\mathfrak{F}\mathfrak{H}$ is a Frobenius group with Frobenius kernel $\mathfrak{F}$ and abelian Frobenius complement $\mathfrak{H}$.
\end{Lemma}

\newpage
\section{Infinite series of right multiplication groups of RCC loops of order $pq$}
\lhead{\slshape 4 \quad  Infinite series of right multiplication groups}
\label{Infinite series of right multiplication groups}

Let $p,q$ be distinct primes.

In this chapter we start to classify the envelops of RCC loops of order $pq$ and we get infinite series of right multiplication groups of RCC loops of order $pq$. 
In \cite{artichiss} the authors determined all envelops of RCC loops of order $2p$. 
If $(G,H,T)$ is the envelop of an RCC loop of order $2p$, then there are three possible types for $G$. First, $G$ can be isomorphic to the wreath product $C_p \wr C_2$. Secondly, $G$ can be isomorphic to a subgroup of the affine group $\Aff$ and lastly, $G$ can be isomorphic to a group $K \times C_2$, where $K$ is an odd order subgroup of $\Aff$ (see \cite[Theorem 5.13]{artichiss}). 
In all these cases $G$ is solvable. 
Moreover, the authors showed that if $(G,H,T)$ is an envelop of an RCC loop of prime order, then $G$ is abelian (see \cite[Proposition 4.3]{artichiss}).

Now let $(G,H,T)$ denote the envelope of an RCC loop of order $pq$. Then $G$ acts faithfully and imprimitively on $H \backslash G$ (see Lemma \ref{Lemma envelop faithful} and \cite[Theorem 3.1]{artichiss}). Thus, $H$ is not a maximal subgroup of $G$. We define $K \lneq G$,\linebreak  such that $H \lneq K$. Without loss of generality, we choose $|G:K|=q$ and $|K:H|=p$.

Now put $T_1:=T \cap K$, $K_1:= \langle T_1 \rangle \leq K$ and $H_1:= H \cap K_1$.

\begin{Lemma}[{\cite[Lemma 5.1]{artichiss}}]
	\label{K1,H1,T1}
	Let the notation be as above. Then \linebreak \((K_1,H_1,T_1)\) is an RCC loop folder of order $p$ with $K_1$ abelian. Also $K_1 \trianglelefteq K$ and $K=HK_1$. Finally, $H_1 \trianglelefteq K.$
\end{Lemma} 

\begin{proof}
	We know that $K$ is the disjoint union of cosets $Ht$ for $t \in T_1$. It follows that $|T_1|=|K:H|=p$ and that $K_1$ is the disjoint union of the cosets $H_1t$ for $t \in T_1$. Since $T$ is $G$-invariant under conjugation, $T_1$ is invariant under conjugation in $K$ and thus, \((K_1,H_1,T_1)\) is an RCC loop folder of order $p$. With \cite[Proposition 4.3]{artichiss} we can conclude that $K_1$ is abelian and clearly, $K_1 \trianglelefteq K$ and $K=HK_1$. Since $K_1 \trianglelefteq K$, we can conclude $H_1 \trianglelefteq H$. Now it follows that $H_1 \trianglelefteq K$, as $K_1$ is abelian and $K=HK_1$.
\end{proof}

Let $C:=\displaystyle\bigcap_{k \in K} H^k \trianglelefteq K$ denote the kernel of the action of $K$ on the cosets of $H$ in $K$. 
In the next two theorems,
we classify the envelop $(G,H,T)$ under the conditions that $K$ is a normal subgroup, $C=\{1\}$ and $H_1=\{1\}$.\\

\begin{Theorem}
	Suppose that $K \trianglelefteq G$, $H_1=\{1\}$, $C=\{1\}$ and $K_1 \lneq C_G(K_1)$. Then $G \cong K \times C_q$ and $K$ is isomorphic to a subgroup of the affine group $\text{\rm Aff}(1,p)$.
\end{Theorem}

\begin{proof}
	Since $K \trianglelefteq G$ and $T$ is $G$-invariant, we have
	\begin{align*}
	K_1^g = \langle (T \cap K)^g \rangle = \langle T \cap K \rangle \text{ for all } g \in G.
	\end{align*}
	Thus, it follows $K_1 \trianglelefteq G$.
	Because $H_1=\{1\}$, it follows with Lemma \ref{K1,H1,T1}  that $|K_1|=|K_1:H_1|=p$. Hence, $K_1$ is cyclic.
	
	Suppose that $h \in C_H(K_1)$ and $k=k_1h_1 \in K=HK_1=K_1H$ with \linebreak $k_1 \in K_1,h_1 \in H_1$. Then we have $h^k=h^{k_1h_1}=h^{h_1} \in H$. Since $k$ was chosen arbitrarily, we have $h^k \in H$ for all $k \in K$ and hence, $h \in H^k$ for all $k \in K$. We conclude $C_H(K_1) \leq C = \{1\}$ and thus, $H \cap C_G(K_1)=C_H(K_1)=\{1\}$.
	Now this yields by our assumption that 
	\begin{align*}
	 |HC_G(K_1)|=|H||C_G(K_1)| > |H||K_1| = |K|.
	\end{align*}
	As $K_1 \trianglelefteq G$, we have $C_G(K_1) \trianglelefteq G$ and therefore, $HC_G(K_1)$ is a subgroup of $G$ containing $K=HK_1$. Together with the inequality above, we can conclude that the index of $HC_G(K_1)$ in $G$ has to be smaller than $|G:K|= q$ and has to divide $q$. Hence, the index is 1 and we have $G=HC_G(K_1)$. Moreover, we have $|C_G(K_1)|=|G:H|=pq=q|K_1|$.
	
	Now there exists $a \in C_G(K_1)$ with $|a|=q$.
	We have 
	\begin{align*}
	K_1 \leq K \cap C_G(K_1) \leq C_G(K_1).
	\end{align*}
	Suppose that $K \cap C_G(K_1) = C_G(K_1)$. 
	Then we conclude that $C_G(K_1) \leq K$.
	But since we also have $H \leq K$, it follows that $HC_G(K_1) \leq K$ and this is a contradiction to $G=HC_G(K_1)$.  Because $|C_G(K_1):K_1|=q$, this implies $K \cap C_G(K_1) = K_1$. As $a \notin K_1$ for reason of order, we have $a \notin K$.
	
	Since we have $K_1 \langle a \rangle \leq C_G(K_1)$ and $|C_G(K_1)|=|K_1||\langle a \rangle|$, it follows that $C_G(K_1)=K_1\langle a \rangle$.
	Hence, $C_G(K_1)=K_1 \times \langle a \rangle$, as $a$ centralizes $K_1$.
	As $K \trianglelefteq G$ and 
	\begin{align*}
	 \langle a \rangle = O_q(C_G(K_1)) \cha C_G(K_1) \trianglelefteq N_G(K_1) = G,
	 \end{align*}
	  it follows from  $|G:K|=q$ that $G = K \times \langle a \rangle.$
	
	By assumption, $K$ acts transitively and faithfully on the set of $H$-cosets in $K$. If $K$ is solvable, it follows from Theorem \ref{HuppertSatz3.6} that $K$  is isomorphic to a subgroup of the affine group Aff($1,p$). We show that $G$ is solvable. 
	We know that $C_G(K_1)=K_1 \times \langle a \rangle$ is solvable, since $K_1$ and $\langle a \rangle$ are solvable.
	Furthermore, $\bigslant{G}{C_G(K_1)} = \bigslant{N_G(K_1)}{C_G(K_1)}$
	 is isomorphic to a subgroup of Aut($K_1$) (see \cite[I, 4.5]{huppert}). Now Aut($K_1$) is abelian because $K_1$ is cyclic. Hence, $\bigslant{G}{C_G(K_1)}$ is solvable and it follows that $G$ is solvable.
\end{proof}

For the proof of the next theorem the following lemma is necessary. 
\begin{Lemma}
	\label{HelpingLemmaP}
	Let $G$ be a group and let $P \trianglelefteq G$ be a cyclic and normal subgroup of $G$ of order $p \in \Prim$ with $C_G(P)=P$. Then $P$ is the only Sylow $p$-subgroup of $G$ and $P$ has a complement $Q$. Further, $G$ acts transitively and faithfully on the set of cosets of $Q$ in G.
\end{Lemma}
\begin{proof}
	By Theorem \cite[I, 4.5]{huppert} we know that $\bigslant{G}{P} = \bigslant{N_G(P)}{C_G(P)}$ is isomorphic to a subgroup of the abelian group Aut($P$) and with Theorem \cite[I, 4.6]{huppert} we know that $|$Aut$(P)|=p-1$. 
	Thus, $|G:P|$ divides $p-1$ and by Zassenhaus' theorem \cite[I, Theorem 18.1]{huppert} $P$ has a complement $Q$ in $G$. Moreover, it follows that $P$ is the only Sylow $p$-subgroup of $G$.
	Clearly, $G$ acts transitively on the set $Q \backslash G$ of cosets of $Q$ in $G$. 
	It remains to show that $G$ acts faithfully on $Q \backslash G$.
	
	It follow from $G=QP$ and $Q \cap P = \{1\}$ that $P$ is a transversal for $Q \backslash G$.
	Now let $x \in G$ with $Qyx=Qy$ for all $y \in P$.
	In particular, we have $Qx=Q$, which implies that $x \in Q$.
	As $P \trianglelefteq G$, we have $x^{-1}yx \in P$ for all $y \in P$. 
	With the fact that $P$ is a transversal for $Q \backslash G$, the equation $Qy=Qyx=Qx^{-1}yx$ for all $y \in P$ implies that $y=x^{-1}yx$ for all $y \in P$.  
	Hence, $x \in C_G(P)=P$. 
	Since $P \cap Q = \{1\}$, we conclude that $x=1$.
	Thus, $G$ acts transitively and faithfully on $Q \backslash G$.
\end{proof}

\begin{Theorem}
	Suppose that $K \trianglelefteq G$, $H_1=1$ and $K_1 = C_G(K_1)$. Then $G$ is isomorphic to a subgroup of the affine group $\text{\rm Aff}(1,p)$.
\end{Theorem}

\begin{proof}
	Again, we have that $K_1$ is cyclic of order $p$ and that $K_1 \trianglelefteq G$. As $C_G(K_1)=K_1$, it follows with Lemma \ref{HelpingLemmaP} that $G$ is a transitive permutation group of degree $p$ with a normal Sylow $p$-subgroup.  
	Then Theorem \ref{HuppertSatz3.6} implies that $G$ is isomorphic to a subgroup of the affine group Aff($1,p$).
\end{proof}

\newpage
\section{Construction of RCC loop folders}

\lhead{\slshape 5 \quad Construction of RCC loop folders}

In this chapter we construct RCC loop folders. 
First, we extend a given RCC loop folder by adding direct and semidirect factors to it and then we investigate if these constructions preserve generating transversals and faithfulness.
Furthermore, we construct invariant generating transversals and we prove that every abelian group with a small enough subgroup has a generating transversal for this subgroup. Then we construct envelops of non-associative RCC loops with these abelian groups. 

\subsection{Direct and semidirect products}

The construction of new RCC loop folders is straight forward as we see in the next theorems.

\begin{Theorem}
	\label{ConstructionGxQ}
	Let $(G,H,T)$ be an RCC loop folder and let $Q$ be a group. Then
	$(G \times Q, H \times \{1\}, T \times Q)$ is an RCC loop folder.
	Additionally, if $(G,H,T)$ is faithful, $(G \times Q, H \times \{1\}, T \times Q)$ is also faithful. Furthermore, if $T$ is a generating transversal, then $G \times Q = \langle T \times Q \rangle$. In particular, if $(G,H,T)$ is an envelop of an RCC loop, then $(G \times Q, H \times \{1\}, T \times Q)$ is an envelop of an RCC loop.
\end{Theorem}

\begin{proof}
	Clearly, $H \times \{1\} \leq G \times Q$ and $(1,1) \in T \times Q$. For all $(g,q) \in G \times Q$ and $(t,s) \in T \times Q$ we have
	\begin{align*}
		(H \times \{1\})^{(g,q)}(t,s)=(H^g \times \{1\})(t,s)=H^gt \times \{s\}.
	\end{align*}
	If $H^gt \times \{q\}=H^gt' \times \{q'\}$
	for some $g \in G $ and $(t,q), (t'q') \in T \times Q$, we can conclude $q=q'$ and $H^gt=H^gt'$. As $T$ is a transversal for the  cosets of every $H^g$, $g \in G$, we have $t=t'$. 
	Moreover, the equation 
	\begin{align*}
		\bigcup_{(t,q) \in T \times Q} (H^gt \times \{q\}) = \bigcup_{t \in T} H^gt \times \bigcup_{q \in Q}\{q\} = G \times Q
	\end{align*}
	holds. 
	Finally, $T \times Q$ is invariant under conjugation in $G \times Q$, as $t^g \in T$ for all $g \in G$ and thus, the first statement follows.
	
	If $(G,H,T)$ is faithful, we have
	\begin{align*}
		\bigcap_{(g,q) \in G \times Q} (H \times \{1\})^{(g,q)} =\bigcap_{(g,q) \in G \times Q} H^g \times \{1\}= \bigcap_{g \in G } H^g \times \{1\} = \{1\} \times \{1\}.
	\end{align*}
	Hence, $(G \times Q, H \times \{1\}, T \times Q)$ is faithful.
	
	Clearly, if $T$ is a generating transversal, then we have $G \times Q = \langle T \rangle \times Q = \langle T \times Q \rangle $.
\end{proof}
	
	Let $(G,H,T)$ be an RCC loop folder and let $Q$ be a group. Suppose that $G$ acts on $Q$ as group of automorphisms. 
	Then we define $T \ltimes Q:=\{(t,q) \mid t \in T, q \in Q\}$ as a subset of $G \ltimes Q$ and we consider $H \ltimes \{1\}$ as subgroup of $G \ltimes Q$.

\begin{Theorem}
	\label{ConstcutionGltimesQ}
	Let $(G,H,T)$ be an RCC loop folder and let $Q$ be a group. 
	If $G$ acts on $Q$ as group of automorphisms, then $(G \ltimes Q, H \ltimes \{1\}, T \ltimes Q)$ is an RCC loop folder.
	Furthermore, if $T$ is a generating transversal, then $G \ltimes Q = \langle T \ltimes Q \rangle$.
\end{Theorem}
	
\begin{proof}
	It follows directly that
	\begin{align*}
		|G \ltimes Q : H \ltimes \{1\}|=|G:H||Q|=|T||Q|=|T \ltimes Q|.
	\end{align*}
	Now suppose that $(t',q'),(t,q) \in T \ltimes Q$ such that $(t',q')(t,q)^{-1} \in H \ltimes \{1\}$ and thus, we have
	\begin{align*}
		(t',q')(t,q)^{-1}=&(t',q')(t^{-1},(q^{-1})^{t^{-1}})=(t't^{-1},(q')^{t^{-1}} (q^{-1})^{t^{-1}} ) \\
		=&(t't^{-1},(q'q^{-1})^{t^{-1}}) \in H \ltimes \{1\}.
	\end{align*}
	This implies that $t't^{-1} \in H$ and $(q'q^{-1})^{t^{-1}}=1$. Since $T$ is a transversal for $H\backslash G$, we have $t=t'$ and
	as $(t^{-1})$ is an automorphism of $Q$, we have $q=q'$.
	Hence, $T \ltimes Q$ is a transversal for $H \ltimes \{1\} \backslash G \ltimes Q$.
	
	Since $t^g \in T$ for all $t \in T$ and $g \in G$, it follows with a simple calculation that $T \ltimes Q$ is invariant under conjugation in $G \ltimes Q$.
	
	Again, if $T$ is a generating transversal, then $G \ltimes Q = \langle T \rangle \ltimes Q = \langle T \ltimes Q \rangle $.
\end{proof}

If $(G,H,T)$ is an envelop of an RCC loop folder, then $(G \ltimes Q, H \ltimes \{1\}, T \ltimes Q)$ does not need to be an envelope of an RCC loop folder in general, because the loop folder does not need to be faithful. The next lemma gives a criterion for this problem for abelian groups.

\begin{Lemma}
	\label{ConstructionFaithful}
	Let $G$ be an abelian group and $H$ a subgroup of $G$. Let $Q$ be a group and suppose that $G$ acts on $Q$ as a group of automorphims. Then we have \begin{align*}
		\text{ \rm core}_{G \ltimes Q}(H \ltimes \{1\})=\{ h \in H \ltimes \{1\} \mid [h, \{1\}\ltimes Q  ]=1 \}.
	\end{align*} 
\end{Lemma}
\begin{proof}
	We write $\hat{G}:=G \ltimes Q$, $\hat{Q}:=\{1\} \ltimes Q$ and $\hat{H}:=H \ltimes \{1\}$. Then we have
	\begin{align*}
   		\hat{C}:=\text{ \rm core}_{\hat{G}}(\hat{H})=\bigcap_{(g,q) \in G \ltimes Q} \hat{H}^{(g,q)}=\bigcap_{(1,q) \in \hat{Q}} \hat{H}^{(1,q)},
	\end{align*}
	since $G$ is abelian. Clearly, we have $\hat{C} \trianglelefteq \hat{G}$ and $\hat{Q} \trianglelefteq \hat{G}$. As $\hat{C} \subseteq \hat{H}$, it follows that $\hat{C} \cap \hat{Q} \subseteq \hat{H} \cap \hat{Q} = \{1\}$. All these facts imply that $[\hat{C},\hat{Q}]=\{1\}$.
	Now let $x  \in \hat{C}$. Then we have $[x,\hat{Q}]=\{1\}$ and hence, $x \in \{ h \in \hat{H} \mid [h, \hat{Q}  ]=1 \}.$ Let $h \in \hat{H}$ with $[h,\hat{Q}]=\{1\}$. This implies that $h=q^{-1}hq$ for all $q \in \hat{Q}$ and thus, $h \in \hat{C}$. In conclusion, the statement follows.
\end{proof}

Let $(G,H,T)$ be an RCC loop folder such that $T$ is a generating transversal and let $Q$ be a group. 
If $G$ acts on $Q$ as group of automorphisms, we can apply Theorem \ref{ConstcutionGltimesQ} and if additionally $G$ is abelian and $[h,\{1\} \ltimes Q] \neq 1$ for all $h \in H \ltimes \{1\}$, then $(G \ltimes Q , H \ltimes \{1\},T \ltimes Q)$ is an envelop of an RCC loop.

Our construction from Theorem \ref{ConstructionGxQ} extends nicely for two given RCC loop folders.

\begin{Theorem}
	\label{ConstructionG1xG2}
	Let $(G_1,H_1,T_1)$ and $(G_2,H_2,T_2)$ be two RCC loop folders. Then the direct product $(G_1 \times G_2, H_1 \times H_2, T_1 \times T_2)$ is an RCC loop folder. If additionally $(G_1,H_1,T_1)$ and $(G_2,H_2,T_2)$ are faithful, then $(G_1 \times G_2, H_1 \times H_2, T_1 \times T_2)$ is also faithful. Moreover, if $T_1$ and $T_2$ are generating transversals, then $T_1\times T_2$ is also a generating transversal.
\end{Theorem}

\begin{proof}
	Since $T_1$ and $T_2$ are transversals for $H_1\backslash G_1$ respectively $H_2 \backslash G_2$, it follows directly that the equation  
\begin{align*}
	\bigcup_{(t_1,t_2) \in T_1 \times T_2} H_1t_1 \times H_2t_2 = \bigcup_{t_1 \in T_1} H_1t_1 \times \bigcup_{t_2 \in T_2} H_2t_2 =G_1 \times G_2
\end{align*} holds.
Moreover, if $H_1t_1 \times H_2t_2 = H_1t'_1 \times H_2t'_2$ for $(t_1,t_2),(t_1',t_2') \in T_1 \times T_2$, then we can conclude that $H_1t_1=H_1t'_1 \text{ and } H_2t_2=H_2t'_2$ and thus, we have $(t_1,t_2)=(t_1',t_2')$. Hence, $T_1 \times T_2$ is a transversal for $H_1 \times H_2 \backslash G_1 \times G_2$. As $(t_1,t_2)^{(g_1,g_2)}=(t_1^{g_1},t_2^{g_2})$ and $T_1$ and $T_2$ are invariant under conjugation in $G_1$ respectively $G_2$, $T_1 \times T_2$ is invariant under conjugation in $G_1 \times G_2$.
Now suppose that  $(G_1,H_1,T_1)$ and $(G_2,H_2,T_2)$ are faithful. Then the following equation holds
\begin{align*}
	\bigcap_{(g_1,g_2) \in G_1\times G_2} H_1^{g_1} \times  H_2^{g_2} = \bigcap_{g_1 \in G_1} H_1^{g_1} \times \bigcap_{g_2 \in G_2}  H_2^{g_2} = \{1\} \times \{1\}
\end{align*}
and thus, $(G_1 \times G_2, H_1 \times H_2, T_1 \times T_2)$ is also faithful.
Furthermore, if $T_1$ and $T_2$ are generating transversals, this implies that 
\begin{align*}
	G_1 \times G_2 = \langle T_1 \rangle \times \langle T_2 \rangle = \langle T_1 \times T_2 \rangle
\end{align*}
and hence, the last statement follows.
\end{proof}

\subsection{Construction of invariant generating transversals}
In the previous theorem we showed that  if we have a direct product of two groups $G_1$, $G_2$, two subgroups $H_1 \leq G_1$ and  $H_2 \leq G_2$ and two generating transversals, then $T_1 \times T_2$ is a generating transversal of $H_1 \times H_2 \backslash G_1 \times G_2$. But if there exists a generating transversal of $H_1 \times H_2 \backslash G_1 \times G_2$, there exists in general no generating transversal for $H_i \backslash G_i$ $(i=1,2)$. This is illustrated in the following example. 

\begin{Example}
Let $D_8:= \langle s,t \mid s^2=t^2=1, (st)^4=1 \rangle$ be the dihedral group of order 8 and set $\widetilde{H}:=\langle s \rangle$. Then $T_1:=\{ 1, t, sts, tsts\}$ and $T_2:=\{1,st,ts,tsts\}$ are the only two $D_8$-invariant transversals for $\widetilde{H} \backslash D_8$ containing 1 but $D_8 \neq \langle T_1 \rangle \neq \langle T_2 \rangle$. 

Let $C_2:=\langle a \mid a^2=1 \rangle$ and define $G:=C_2 \times D_8$ and $H:=\{1\} \times \langle s \rangle$. 
The sets $C_2 \times T_1$ and $C_2 \times T_2$ are transversals for $H \backslash G$ but they do not generate $G$. 
However, we can use the fact that $H(1,st) \neq H(a,t)$ in order to construct a generating transversal for $H$ in $G$:
\begin{align*}
 \begin{array}{llllll}
S:=\{ &(1,1), & (1,st), & (1,ts),& (1,tsts), & \\
&(a,1), & (a,t), & (a,sts),& (a,tsts) & \}. \\
\end{array}\\
\end{align*} 
As $T_1$ and $T_2$ are $G$-invariant transversals for $D_8$ and $C_2$ is abelian, $S$ is a $G$-invariant transversal for $H \backslash G$ and 
since we have 
\begin{align*}
(a,1)(1,st)(a,t)=(1,s),
\end{align*}
 we can conclude that $G = \langle S \rangle$. 
 Furthermore, we obtain that $(G,H,S)$ is a faithful RCC loop folder, as $|H|=2$ and $H$ is not normal in $G$.
 Thus, $(G,H,S)$ is an envelop of an RCC loop.
\end{Example}

In the above example we constructed with two invariant transversals, which together generate the whole group, a generating transversal. The following theorem shows that this a general phenomenon.

\begin{Theorem}
	Let $G$ be a group and $H \leq G$. Let $\mathcal{T}$ be the set of \linebreak $G$-invariant transversals of $H \backslash G$ containing 1. Suppose that there exists a set $\mathcal{S} \subseteq \mathcal{T}$ such that 
	\begin{align*}
		\big\langle \displaystyle\bigcup_{S \in \mathcal{S}} S \big\rangle =G.
	\end{align*}
	Let $Q$ be an abelian group such that $|Q| \geq |\mathcal{S}|$. Then there exists a \linebreak $G \times Q$-invariant generating transversal $T$ for $H \times \{1\}$ in $G \times Q$ with $1 \in T$.
\end{Theorem}

\begin{proof} We denote the set $\mathcal{S}$ by $\{S_1, \dotsc,S_m\}$ and the abelian group $Q$ by $\{q_1,  \ldots,q_n\}$.
Now we can construct the transversal for $H \times \{1\}$ in $G \times Q$ as follows:
	\begin{align*}
		T:=\bigcup_{i=1}^{m} S_i \times \{q_i\} \;\cup \bigcup_{j=m+1}^n S_1 \times \{q_j\}.
	\end{align*}
	First, we prove that $T$ is a transversal for $H \times \{1\} \backslash G \times Q$. As $|G:H|=|S_i|$ for all $1 \leq i \leq m$, we obtain that
	\begin{align*}
		|T|&= \sum_{i=1}^{m} |S_i||\{q_i\}|+\sum_{j=m+1}^{n} |S_1||\{q_j\}|=|G:H|n=|G:H||Q|\\
		&=|G \times Q : H \times \{1\}|.
	\end{align*}
	Moreover, let $(s',q_i),(s,q_j) \in T$ such that $(s',q_i)(s,q_j)^{-1} \in H \times \{1\}$. 
	Then it follows that $q_i=q_j$ and thus, $i=j$. This yields that $s,s' \in S_k$ for $k=1$ or $k=i$. 
	As $S_k$ is a transversal for $H \backslash G$ and $s's^{-1} \in H$, we obtain that $s=s'$. 
	Hence, $(s',q_i)=(s,q_j)$ and we conclude that  $T$ is a transversal for $H \times \{1\} \backslash G \times Q$.
	Since $S_i$ is $G$-invariant for every $1 \leq i \leq m$ and $Q$ is abelian, it follows that $S_i \times \{q_j\}$ is $G \times Q$-invariant. Therefore, $T$ is $G \times Q$-invariant. 
	
	Since $\langle T \rangle$ contains $S \times 1$ for all $S \in \mathcal{S}$, it follows that $G \times \{1\} \leq \langle T \rangle$. Clearly, $\langle T \rangle$ contains $\{1\} \times Q$ and hence, we have $\langle T \rangle = G \times Q$. In conclusion, $T$ is a $G \times Q$-invariant generating transversal for $H \times \{1\}$ in $G \times Q$ with $1 \in T$.
\end{proof} 

If additionally to the assumptions of the previous theorem $G$ acts faithful on $H \backslash G$, Theorem \ref{ConstructionGxQ} implies that $(G \times Q, H \times \{1\}, T)$ is an envelop of an RCC loop. Thus, if $G$ acts faithful on $H \backslash G$ and if we can generate $G$ with the set of $G$-invariant transversals for $H \backslash G$, we can construct an envelop of an RCC loop.

\subsection{Construction of generating transversals for abelian groups}
	
	In the following section we investigate the existence of invariant generating transversals for certain subgroups in abelian groups.
	First, we show that for every abelian $p$-group $G$ with $p \in \Prim$ and every subgroup $H$ of $G$ of index larger than the minimal size of a generating set, there exists a generating transversal for $H \backslash G$ containing 1. 
	Then we construct a generating transversal for a subgroup $H$ in an arbitrary abelian group $G$ out of a generating transversal for a Sylow subgroup of $H$.
	Moreover, with these abelian groups we can construct envelops of non-associative RCC loops.
	
	We first define and prove some properties of the minimal size of a generating set.
  
\begin{Def}
	Let $G$ be a group. We call the minimal size of a generating set
	\begin{align*}
		\rk(G):=\min\{ |S| \mid S \subseteq G, G = \langle S \rangle \}
	\end{align*}
	the rank of the group $G$.
\end{Def}

\begin{Lemma}
	\label{LemmaRank2}
	Let $G_1$ and $G_2$ be groups such that $\gcd(|G_1|,|G_2|)=1$. Then we have $\rk(G_1 \times G_2)=\max\{\rk(G_1),\rk(G_2)\}$.
\end{Lemma}
\begin{proof}
	Let $n:=\rk(G_1 \times G_2)$. Then $G_1 \times G_2$ can be generated by $n$ elements and thus, both $G_1$ and $G_2$ can be generated by $n$ elements as they are factor groups of $G_1 \times G_2$. 
	It follows that $\rk(G_i) \leq \rk(G_1 \times G_2)$ for $i=1,2$ and hence, we obtain
	\begin{align*}
	\max\{\rk(G_1),\rk(G_2)\} \leq \rk(G_1 \times G_2).
	\end{align*}
	
	Conversely, suppose, without loss of generality, that $m:=\rk(G_1) \geq \rk(G_2)$. Let $\{t_1, \dotsc, t_m\}$ be a minimal generating set for $G_1$ of size $m$ and let \linebreak
	$\{s_1, \ldots, s_m\}$ be a generating set for $G_2$. Now we show that 
	\begin{align*}
	S:=\{(t_i,s_i) \mid 1 \leq i \leq m\}
	\end{align*}
	 is a generating set for $G_1 \times G_2$. Since  $\gcd(|G_1|,|G_2|)=1$, there exist \linebreak
	  $a_i \in \N$ such that $(t_i,s_i)^{a_i}=(t_i,1)$ and there exist $b_i \in \N$ such that \linebreak
	  $(t_i,s_i)^{b_i}=(1,s_i)$ for every $1 \leq i \leq n$. This implies that $\langle S \rangle = G_1 \times G_2$ and hence, we have 
	\begin{align*}
	\rk(G_1 \times G_2) \leq |S| &=\max\{\rk(G_1),\rk(G_2)\}. &&\qedhere
	\end{align*}
\end{proof}

Let $G$ be an abelian group and let $p_1, \dotsc ,p_n$ be the distinct prime divisors of $G$. Then $G_i:=O_{p_i}(G) \leq G$ is the only Sylow $p_i$-group for all $1 \leq i \leq n$. It follows from \cite[5.1.4 Satz]{kurzweil} that $G$ is the inner direct product of $G_1, \dotsc, G_n$ and hence, without loss of generality, we can assume that 
\begin{align*}
G = G_1 \times G_2 \times \dotsb \times G_n.
\end{align*}
Let $H \leq G$ be a subgroup of $G$. Then $H_i:=O_{p_i}(H) \leq G_i$ and again, we can assume that 
\begin{align*}
H = H_1 \times H_2 \times \dotsb \times H_n.
\end{align*}

Now we transfer the result of Lemma \ref{LemmaRank2} to an arbitrary abelian group.

\begin{Lemma}
	\label{LemmaRankn}
	Let $G$ be an abelian group and let $p_1, \dotsc ,p_n$ be the distinct prime divisors of $G$.
	Assume that $G=G_1 \times \dotsb \times G_n$ with $G_i:=O_{p_i}(G)$ for all $1 \leq i \leq n$.
	Then $\rk(G)=\max\{\rk(G_i) \mid 1 \leq i \leq n\}$.
\end{Lemma}

\begin{proof}
	We prove this statement by induction on $n$.
	The base case for $G=G_1$ is trivial. 
	Thus, assume that the statement is true for groups with $n-1$ distinct prime divisors and suppose that $G= G_1 \times \dotsb \times G_n$.
	Then the induction hypothesis applied to $\widetilde{G}:=G_1 \times \dotsb \times G_{n-1}$ implies that 
	\begin{align*}
		\rk(\widetilde{G})=\max\{\rk(G_i) \mid 1 \leq i \leq n-1\}.
	\end{align*}
	Now consider $G=\widetilde{G} \times G_n$. Since we have  $\gcd(|\widetilde{G}|,|G_n|)=1$, we can apply Lemma \ref{LemmaRank2} and we conclude that 
	\begin{align*}
	\rk(G)= \max\{\rk(\widetilde{G}),\rk(G_n)\}= \max\{\rk(G_i) \mid 1 \leq i \leq n\}.
	\end{align*}
	Now the statement holds for every abelian group.
\end{proof}

If $G$ is a cyclic group, it is easy to see that there exists a generating transversal for every proper subgroup of $G$. Nevertheless, we construct such a transversal in the next remark.

\begin{Rem}
	\label{cyclicgenerating}
	Let $G$ be a cyclic group and suppose that $H \lneq G$ is a proper subgroup of $G$. Let $g \in G$ with $G=\langle g \rangle$ and let $n:=|G|$ and $d:=|H|$. Then $H=\langle g^{\frac{n}{d}} \rangle$. Thus, $T:= \{1,g, \dotsc ,g^{\frac{n}{d}-1}  \}$ is a transversal for $H$ in $G$ and as $g \in T$, it follows that $T$ is a generating transversal.
\end{Rem}

In the following two theorems we examine abelian $p$-groups with a subgroup $H$ and we see that there exists a transversal $T$ for $H\backslash G$ containing 1 such that either $\langle T \rangle =G$ or $T - \{1\}$ is a minimal generating set for $\langle T \rangle$.

For these theorems we need the following decomposition of an abelian group:\\
Let $G$ be an abelian group. Then it follows from the fundamental theorem of finite abelian groups that there exist cyclic groups $C_{m_i} \leq G$ for all $1 \leq i \leq r$ such that 
\begin{align*}
G \cong C_{m_1} \times C_{m_2} \times \dotsb \times C_{m_r},
\end{align*}
where $m_r > 1$ and $m_j$ divides $m_{j-1}$ for all $1 < j \leq r $.

\begin{Theorem}
	\label{pabeliangenerating}
	Let $G$ be an abelian $p$-group with $p \in \Prim$. Suppose that $H \leq G$ is a subgroup of $G$ such that $|G:H| > \rk(G)$. Then there exists a generating transversal $T$ for $H  \backslash G$ with $1 \in T$.
\end{Theorem}
\begin{proof}
	We distinguish between two cases in this proof.	
	
	\textbf{Case 1:}
	Suppose that there exist $C_{m_i} \leq G$ for $1 \leq i \leq r$ such that 
	\begin{align*}
	G \cong C_{m_1} \times C_{m_2} \times \dotsb \times C_{m_r}
	\end{align*}
	and $C_{m_i} \leq H$ for some $1 \leq i \leq r$.
		
	Note that the generators of these cyclic groups are a minimal generating set of size $r$. Thus, it follows from Burnside's basis theorem \cite[III, Satz 3.15]{huppert} that  $r=\rk(G)$.	
	
	We prove this case by induction on the order of $G$. The base case is trivial. 
	Now we show the induction step. Therefore, suppose that the statement holds for every abelian $p$-group of order less than $|G|$.

	We set $U:=C_{m_i}$ and let $u$ denote a generator of $U$. 
	Since $U$ is a direct factor of $G$, there exists a complement $\widetilde{G}$ to $U$ in $G$. 
	We set $\widetilde{H}:=\widetilde{G} \cap H$. Clearly, $\widetilde{H}$ is a complement to $U$ in $H$ and thus, without loss of generality, we consider
	\begin{align*}
		G=\widetilde{G} \times U \quad \text{ and } \quad  H =  \widetilde{H} \times U.
	\end{align*}
	Because we have
	\begin{align*}
		\widetilde{G} \cong C_{m_1} \times \dotsb \times C_{m_{i-1}} \times C_{m_{i+1}} \times \dotsb \times C_{m_r},
	\end{align*}
it follows that
\begin{align*}
\rk(\widetilde{G}) = r-1 < r =\rk(G).
\end{align*}
Since we have $\big|\widetilde{G}\big| < |G|$ and 
\begin{align}
\label{Ungleichung1}
\big| \widetilde{G}: \widetilde{H} \big| = |G:H| > \rk(G) > \rk(\widetilde{G}), 
\end{align}
we can apply the induction hypothesis to $\widetilde{G}$ and hence, there exists a generating transversal $\widetilde{T}$ for $\widetilde{H} \backslash \widetilde{G}$ with $1 \in \widetilde{T}$. 
Moreover, with \eqref{Ungleichung1} the equation 
\begin{align}
	\label{OrdnungTschlange}
	\big| \widetilde{T} - \{1\} \big| = \big| \widetilde{G}: \widetilde{H} \big|-1 > \rk(\widetilde{G}).
\end{align}
holds.
Suppose that $\widetilde{T}- \{1\}$ is a minimal generating set for $\widetilde{G}$. Then it follows from Burnside's basis theorem
\cite[III, Satz 3.15]{huppert} that $\big| \widetilde{T} - \{1\} \big| = \rk (\widetilde{G})$, but this is a contradiction to \eqref{OrdnungTschlange}.
Thus, $\widetilde{T}- \{1\}$ is not a minimal generating set and there exists a $t \in \widetilde{T} - \{1\}$ such that $t=t_1 \dotsb t_k$ for $t_1, \dotsc ,t_k \in \langle \widetilde{T} - \{1,t\} \rangle $.

We obtain that $\widetilde{T} \times \{1\}$ is a transversal for $H \backslash G$ because we have
\begin{align}
\label{T1}
\bigcup_{(t,1) \in \widetilde{T} \times \{1\}} (\widetilde{H} \times U) (t,1)=\bigcup_{t \in \widetilde{T}} \widetilde{H}t \times U = \widetilde{G} \times U = G
\end{align}
and 
\begin{align}
\label{T2}
|\widetilde{T} \times \{1\}|=|\widetilde{T}|=|G:H|.
\end{align}

We set
\begin{align*}
T:=(\widetilde{T} \times \{1\} - \{(t,1)\}) \cup \{(t,u)\}.
\end{align*}
$T$ is a transversal for $H \backslash G$ since $(t,1)$ and $(t,u)$ are lying in the same coset of $H$ in $G$. 
Now we prove that $T$ is a generating transversal for $H \backslash G$.  Recall that $t=t_1 \dotsb t_k$ for $t_1, \dotsc ,t_k \in  \langle \widetilde{T} - \{1,t\} \rangle $. We know that $((t_1 \dotsb t_k)^{-1},1) \in \langle T \rangle$, as $(t_1,1), \dotsc , (t_k,1) \in \langle T \rangle $. Hence, we obtain $(1,u)=((t_1 \dotsb t_k)^{-1},1)(t,u) \in \langle T \rangle$ and this implies that $(t,1)=(t,u)(1,u^{-1}) \in \langle T \rangle$. We conclude that $\langle T \rangle \geq \langle \widetilde{T} \rangle \times \langle u \rangle = \widetilde{G} \times U = G$ and thus, $T$ is a generating transversal for $H \backslash G$ containing 1.

\textbf{Case 2:}
Suppose that for every decomposition 
\begin{align*}
G \cong C_{m_1} \times C_{m_2} \times \dotsb \times C_{m_r},
\end{align*}
where $C_{m_i} \leq G$ for $1 \leq i \leq r$, we have  $C_{m_i} \nleq H$ for all $1 \leq i \leq r$.

If $r=1$, then the statement follows directly from Remark \ref{cyclicgenerating}. Hence, suppose that $r \geq 2$.

Now we consider one decomposition of $G$. Therefore, suppose that 
\begin{align*}
G \cong C_{m_1} \times C_{m_2} \times \dotsb \times C_{m_r},
\end{align*}
where $C_{m_i} \leq G$ for $1 \leq i \leq r$.
Let $a_i$ be a generator of $C_{m_i}$ for every $1 \leq i \leq r$.
Thus, we have $G=\langle a_1, \dotsc, a_r \rangle$ and the assumption that $H$ does not contain a cyclic direct factor of $G$ implies that $a_i \notin H$ for all $1 \leq i \leq r$.

Assume that $Ha_i=Ha_j$ with $1 \leq i,j \leq r$ and $i \neq j$. Without loss of generality, we assume that $|a_i|\geq|a_j|$.
Now we show that 
\begin{align}
\label{Zerlegung2}
G \cong \langle a_1 \rangle \times \dotsb \times \langle a_{i-1} \rangle \times \langle a_ia_j^{-1} \rangle \times \langle a_{i+1} \rangle \times \dotsb \times \langle a_r \rangle.
\end{align}
Since we have $ \displaystyle C_{m_k} \cap \prod_{l=1,l\neq k}^r  C_{m_l}= \{1\}$
for all $1 \leq k \leq r$, we only need to show that 
\begin{align}
\label{Inter}
\displaystyle \langle a_ia_j^{-1} \rangle \cap \prod_{l=1,l\neq i}^r  \langle a_k \rangle = \{1\}.
\end{align}
Suppose that $(a_ia_j^{-1})^n\in \displaystyle \langle a_ia_j^{-1} \rangle \cap \prod_{l=1,l\neq i}^r  \langle a_k \rangle$ for some $n \in \N$.
Then it follows that
\begin{align*}
\displaystyle a_i^n \in \displaystyle \langle a_i \rangle \cap \prod_{l=1,l\neq i}^r  \langle a_k \rangle = \{1\}.
\end{align*}
Thus, the intersection in \eqref{Inter} is trivial.
Moreover, the fact $|a_i|=|a_ia_j^{-1}|$ yields that 
\begin{align*}
 |G|= \prod_{k=1}^{r}|\langle a_k \rangle| = (\prod_{k=1, k \neq i }^{r} |\langle a_k \rangle|)\;|\langle a_i a_j^{-1}\rangle|
\end{align*}
and this implies that
\begin{align*}
G= \langle a_1 \rangle \dotsb\langle a_{i-1} \rangle\langle a_ia_j^{-1} \rangle  \langle a_{i+1} \rangle \dotsb \langle a_r \rangle.
\end{align*}
Hence, the equation in \eqref{Zerlegung2} holds.

Recall that we assumed $Ha_i=Ha_j$. Thus, it follows that $\langle a_ia_j^{-1} \rangle \leq H$. But this is a contradiction to the assumption that $H$ does not contain a cyclic factor of $G$ and hence, the generators $a_i$ and $a_j$ of $G$ lie in different cosets of $H$ in $G$. Since $i$ and $j$ were chosen arbitrarily, every generator $a_i$ of $G$ lies in a different coset of $H$ in $G$ and thus, we can construct a generating transversal for $H \backslash G$ containing 1.
\end{proof}

Now we consider the opposite case to Theorem \ref{pabeliangenerating} where $G$ is an abelian $p$-group such that the index of the subgroup $H$ is less or equal to the minimal size of a generating set of $G$.

\begin{Theorem}
		\label{pAbelianMinGenerating}
		Let $G$ be an abelian $p$-group with $p \in \Prim$. Suppose that $H \leq G$ is a subgroup of $G$ such that $|G:H| \leq \rk(G)$. Then there exists a transversal $T$ for $H  \backslash G$ with $1 \in T$ such that $\rk(\langle T- \{1\} \rangle)=|T|-1$, i.e. $T-\{1\}$ is a minimal generating set for $\langle T \rangle$.
\end{Theorem}
\begin{proof}
	The structure of this proof is very similar to the structure of the previous proof. 
	Hence, a few arguments used here are described in more detail in the previous proof. 	
	Similar to the previous proof, we distinguish between two cases.

	\textbf{Case 1:}
	Suppose that there exist $C_{m_i} \leq G$ for $1 \leq i \leq r$ such that 
	\begin{align*}
	G \cong C_{m_1} \times C_{m_2} \times \dotsb \times C_{m_r}
	\end{align*}
	and $C_{m_i} \leq H$ for some $1 \leq i \leq r$.
	
	We prove this case by induction on the order of $G$. The base case is trivial. Therefore, suppose that the statement holds for every abelian $p$-group of order less than $|G|$.
	
	We set $U:=C_{m_i}$. Once more, it follows that there exists a complement $\widetilde{G}$ to $U$ in $G$ and that $\widetilde{H}:=\widetilde{G} \cap H$ is a complement to $U$ in $H$.
	Thus, without loss of generality, we consider
	\begin{align*}
	G=\widetilde{G} \times U \quad \text{ and } \quad  H =  \widetilde{H} \times U.
	\end{align*}
	Furthermore, we have $\rk(G)=r$ and $\rk(\widetilde{G})=r-1$.
	
	\textbf{Case 1a:}
	Suppose that $|G:H|<\rk(G)$.\\
	This yields that 
	\begin{align*}
		|\widetilde{G}:\widetilde{H}|=|G:H|\leq \rk(G)-1=\rk(\widetilde{G}).
	\end{align*}
	Since we have $|\widetilde{G}|<|G|$, we can apply the induction hypothesis and thus, there exists a transversal $\widetilde{T}$ for $\widetilde{H} \backslash \widetilde{G}$ with $1 \in \widetilde{T}$ such that 
	\begin{align*}
		\rk(\langle \widetilde{T} - \{1\} \rangle )= |\widetilde{T}|-1.
	\end{align*}
	As we have seen in the proof of Theorem \ref{pabeliangenerating} from the equations \eqref{T1} and \eqref{T2}, the set $T:=\widetilde{T} \times \{1\}$ is a transversal for $H \backslash G$ with $1 \in T$. 
	Moreover, we obtain
	\begin{align*}
		\rk(\langle T - \{1\} \rangle )=\rk(\langle \widetilde{T} - \{1\} \rangle ) = |\widetilde{T}|-1 = |T|-1.
	\end{align*} 
	Hence, the statement follows by induction.
	
	\textbf{Case 1b:}
	Suppose that $|G:H|=\rk(G)$.\\
	This implies that 
	\begin{align*}
	|\widetilde{G}:\widetilde{H}|=|G:H|=\rk(G)=r> r-1 =\rk(\widetilde{G}).
	\end{align*}
	Now it follows from Theorem \ref{pabeliangenerating} that there exists a transversal $\widetilde{T}$ for $\widetilde{H} \backslash \widetilde{G}$ with $1 \in \widetilde{T}$ such that $\widetilde{G}=\langle \widetilde{T} \rangle$. 
	This yields that
	\begin{align*}
	|\widetilde{T}-\{1\}|=|\widetilde{G}:\widetilde{H}|-1=\rk(\widetilde{G})
	\end{align*} and hence, $\widetilde{T} - \{1\}$ is a minimal generating set  for $\widetilde{G}$. 
	We set $T:=\widetilde{T} \times \{1\}$. 
	Then once more, $T$ is a transversal for $H\backslash G$ with $1 \in T$. 
	Moreover, $\widetilde{T} \times \{1\}$ is a minimal generating set for $\widetilde{G} \times \{1\}$ and hence, $T - \{1\}$ is a minimal generating set for $\langle T \rangle$.
	We conclude that $T$ is a transversal for $H\backslash G$ with $1 \in T$ and $\rk(\langle T- \{1\} \rangle) = |T|-1$.
	
	\textbf{Case 2:} 
	Suppose that for every decomposition 
	\begin{align*}
	G \cong C_{m_1} \times C_{m_2} \times \dotsb \times C_{m_r},
	\end{align*}
	where $C_{m_i} \leq G$ for $1 \leq i \leq r$, we have  $C_{m_i} \nleq H$ for all $1 \leq i \leq r$.
	
	Now we consider one decomposition of $G$. Therefore, suppose that 
	\begin{align*}
	G \cong C_{m_1} \times C_{m_2} \times \dotsb \times C_{m_r},
	\end{align*}
	where $C_{m_i} \leq G$ for $1 \leq i \leq r$.
	Let $a_i$ be a generator of $C_{m_i}$ for every $1 \leq i \leq r$.
	Thus, we have $G=\langle a_1, \dotsc, a_r \rangle$. 
	Again, it follows that every generator $a_i$ of $G$ lies in a different coset of $H$ in $G$.
	But this is a contradiction to the assumption that $|G:H| \leq \rk(G)=r$ and thus, Case 2 does not occur.
\end{proof}

Next, with the results established before in this section we can prove our main theorem.
\begin{Theorem}
	\label{MainTheoremAbelian}
	Let $G$ be an abelian group, let $p_1, \dotsc p_n$ be the distinct prime divisors of $G$ and let $H \leq G$. Assume that 
	\begin{align*}
	G = G_1 \times \dotsb \times G_n \quad \text{ and } \quad  H = H_1 \times \dotsb \times H_n,
	\end{align*}
	with $G_i:=O_{p_i}(G)$ and $H_i:=O_{p_i}(H)$.
	If
	\begin{align*}
	\max\{ |G_i  : H_i| \mid 1 \leq i \leq n \} > \rk(G),
	\end{align*}
	then there exists a generating transversal for $H\backslash G$ containing 1.
\end{Theorem}
\begin{proof}
	Without loss of generality, we assume that 
	\begin{align*}
	|G_1:H_1|=\max\{ |G_i  : H_i| \mid 1 \leq i \leq n \}
	\end{align*}
	and we set $\widetilde{G}:= G_2 \times \dotsb \times G_n$ and $\widetilde{H}:= H_2 \times \dotsb \times H_n$.
	Now we consider $\widetilde{H}$ as subgroup of $\widetilde{G}$, $G=G_1 \times \widetilde{G}$ and $ H=H_1 \times \widetilde{H}.$
	
	We know from Lemma \ref{LemmaRankn} that
	\begin{align*}
	 m:=|G_1:H_1|&=\max\{ |G_i  : H_i| \mid 1 \leq i \leq n \} > \rk(G)\\ &= \max\{ \,\; \rk(G_i) \,\;\mid 1 \leq i \leq n \} \geq \rk(G_1).
	\end{align*}
	
	Since $G_1$ is an abelian $p_1$-group with $|G_1:H_1| > \rk(G_1)$, it follows from Theorem \ref{pabeliangenerating} that there exists a transversal $T_1=\{t_1, \dotsc, t_m\}$ for $H_1 \backslash G_1$ with $t_1=1$ and $G_1=\langle T_1 \rangle$.
	
	We define $K:= H_1 \times \widetilde{G}$. Then we have $H \leq K \leq G$. Now we construct a generating transversal for $K\backslash G$ containing 1. It follows from the assumptions that
	\begin{align*}
	|G:K|=|G_1:H_1||\widetilde{G}:\widetilde{G}|=|G_1:H_1|=|T_1|=m
	\end{align*}
	and from Lemma \ref{LemmaRankn} that
	\begin{align*}
	k:=rk(\widetilde{G})=&\max\{ \rk(G_i) \mid 2 \leq i \leq n  \} \leq \max\{ \rk(G_i) \mid 1 \leq i \leq n  \}\\ <& \max\{ |G_i  : H_i| \mid 1 \leq i \leq n \}=|G_1:H_1|=m
	\end{align*}
	Let $S$ be a minimal generating set of $\widetilde{G}$ such that $|S|=k$. Then we write $S \cup \{1\}:=\{s_1, \dotsb,s_{k+1}\}$ with $s_1=1$. As $|S \cup \{1\}|=k+1 \leq m$, we set
	\begin{align*}
	R:= \bigcup_{i=1}^{k+1} (t_i,s_i) \cup \bigcup_{j=k+2}^m (t_j,s_1).
	\end{align*}
	In the following, we show that $R$ is a generating transversal for $K \backslash G$ with $1 \in R$. As $t_1=1$ and $s_1=1$, we have $(1,1) \in R$. Moreover, we have $|R|=m=|G:K|$.
	Suppose that $(t_i,s_j),(t_k,s_l) \in R$ such that $(t_i,s_j)(t_k,t_l)^{-1} \in H_1 \times \widetilde{G}$. Then we have $t_it_k^{-1} \in H_1$ and as $T_1$ is a transversal for $H_1 \backslash G_1$, it follows that $i=k$. This implies that $j=l$.
	We conclude that $R$ is a transversal for $K\backslash G$. 
	
	The fact $\gcd(|G_1|,|\widetilde{G}|)=1$ yields that for every $(t,s) \in R$ there exist $a \in \N$ and $b \in \N$ such that $(t,s)^a=(1,s)$ and $(t,s)^b=(t,1)$. Hence, it follows that 
	\begin{align*}
	\langle R \rangle \geq \langle T_1 \rangle \times \langle S \rangle = G_1 \times \widetilde{G} =G
	\end{align*}
	and thus, $R$ is a generating transversal for $K \backslash G$ with $1 \in R$.
	
	Let $V$ be a transversal for $H \backslash K$ with $1 \in V$. Then we set 
	\begin{align*}
	T:= VR.
	\end{align*}
	Clearly, $1 \in T$. 
	The facts that  $R$ is a transversal for $K \backslash G$ and $V$ is a transversal for $H \backslash K$ imply that
	\begin{align*}
	\bigcup_{t \in T} Ht = \bigcup_{r \in R} (\bigcup_{s \in V} Hs)r = \bigcup_{r \in R} Kr =G
	\end{align*} 
	and 
	\begin{align*}
	|T|=|V||R|= |G:K||K:H| = |G:H|.
	\end{align*}	
	Hence, $T$ is a transversal for $H \backslash G$.
	Since $1 \in V$, we have $R \subseteq T$ and it follows that $\langle T \rangle \geq \langle R \rangle = G$. This implies that $T$ is a generating transversal for $H \backslash G$ with $1 \in T$. 
\end{proof}

Note that in Theorem \ref{envelop} we proved that an RCC loop folder $(G,H,T)$, where $G$ is non abelian, $G$ acts faithfully on $H \backslash G$ and $\langle T \rangle = G$, is an envelop of a non-associative RCC loop.

The next remark describes a construction of envelops of non-associative RCC loops with abelian groups. Hence, we obtain an infinite series of right multiplication groups of non-associative RCC loops.

\begin{Rem}
	Let $G$ be an abelian group and let $H$ be a subgroup of $G$ such that  
	\begin{align*}
		\max\{|O_p(G):O_p(H)| \mid p \text{ prime divisor of }G \} > \rk(G).
	\end{align*}
	Then it follows from Theorem \ref{MainTheoremAbelian} that there exists a generating transversal $T$ for $H \backslash G$. Clearly, $(G,H,T)$ is an RCC loop folder.
	Now we choose a group $Q$ such that $G$ acts on $Q$ as group of automorphisms and such that we have $[h,\{1\} \ltimes Q] \neq 1$ for all $h \in H \ltimes \{	1\}$. 
	Then it follows from Theorem \ref{ConstcutionGltimesQ} and Lemma \ref{ConstructionFaithful} that
	\begin{align*}
	(G \ltimes Q, H \ltimes \{1\}, T \ltimes Q)
	\end{align*}
	 is an envelop of a non-associative RCC loop.
\end{Rem}

\pagebreak

It is worth noting that there is an analog construction of the above construction for quotient groups. 
Instead of extending an abelian RCC loop folder with a generating transversal by adding a semi-direct factor, we can also extend a generating transversal from an abelian quotient to the whole group.

\begin{Cor}
	\label{CorAbelianGrp}
	Let $G$ be a group and let $H$ be a subgroup of $G$. 
	Let $Q$ be a normal subgroup of $G$ such that $G / Q$ is abelian, $H \cap Q = \{1\}$ 
	and 
	\begin{align*}
	\max\{|O_p(G/Q):O_p(HQ/Q)| \mid p \text{ prime divisor of } G/Q \} > \rk(G/Q).
	\end{align*}
	Then there exists a $G$-invariant generating transversal $T$ for $H \backslash G$ with $1 \in T$.
\end{Cor}
\begin{proof}
	Since $G / Q$ is abelian and
	\begin{align*}
		\max\{|O_p(G/Q):O_p(HQ/Q)| \mid p \text{ prime divisor of } G/Q \} > \rk(G/Q),
	\end{align*}
	 it follows from Theorem \ref{MainTheoremAbelian} that there exists a transversal 
	  \begin{align*}
	  	 \widetilde{T}=\{Qt_1, \dotsc, Qt_n\}
	  \end{align*}
 	for $(HQ / Q) \big\backslash (G/Q)$ with $t_1=1$ and $G /Q = \big\langle \widetilde{T} \big\rangle$. 
 	
	We set $\hat{T}:=\{t_1, \dotsc, t_n\}$ and $T:=Q\hat{T}$.
	We show that $T$ is transversal for $H \backslash G$. Clearly, $1 \in T$ and we have
	\begin{align*}
	\bigcup_{t \in T} Ht = \bigcup_{\hat{t} \in \hat{T}} HQ\hat{t} = G.
	\end{align*}
	Furthermore, we obtain
	\begin{align*}
	\left| T \right| = \left| Q \right|   \big| \hat{T} \big| = 	\left| Q \right| \frac{\left|G  \right|}{\left| Q\right|} \frac{\left| Q \right|}{\left| H \right|\left| Q \right|} = \frac{\left| G \right|}{\left| H\right|}.
	\end{align*}
	Thus, $T$ is a transversal for  $H \backslash G$ containing 1. 
	Clearly, $T$ is a generating transversal. 
	Therefore, it remains to show that $T$ is $G$-invariant. Let $g \in G$ and $qt \in T$ with $q \in Q$ and $t \in \hat{T}$. 
	As $Q \trianglelefteq G$, we have $g^{-1}qg \in Q$ and it follows from the fact that $\widetilde{T}$ is $G/Q$-invariant that
	\begin{align*}
	g^{-1}qtg=g^{-1}qgg^{-1}tg \in Qg^{-1}tg = g^{-1}(Qt)g \in \widetilde{T}.
	\end{align*} 
	This yields that $g^{-1}qtg \in Q\hat{T}=T$.
\end{proof}

\pagebreak

\begin{Rem}
	Let the assumption and notation be as in Corollary \ref{CorAbelianGrp}. Suppose that $C_H(Q)=\{1\}$. We show that $G$ acts faithfully on $H \backslash G$.

	Let $h \in \text{ \rm core}_G(H)= \displaystyle\bigcap_{g \in G } H^g$ and let $q \in Q$. Then we have $h=q^{-1}h'q$ for some $h' \in H$. It follows from $Q \trianglelefteq G$ that 
	\begin{align*}
		hh'^{-1}=q^{-1}h'qh'^{-1} \in Q \cap H = \{1\}.
	\end{align*}
	This implies that $h=h'$. Since $q$ was arbitrary, we have $hq=qh$ for all $q \in Q$ and hence, $h\in C_H(Q)$.
	Now it follows that
	\begin{align*}
	\text{ \rm core}_G(H)= \bigcap_{g \in G } H^g \leq C_H(Q)= \{1\}
	\end{align*}
	 and thus, $(G,H,T)$ is a faithful RCC loop folder with $G=\langle T \rangle$. In conclusion, $(G,H,T)$ is an envelop of an RCC loop.
\end{Rem}

With the result above, we can now improve our results from Lemma \ref{GStrichSchnitH}. 

\begin{Rem}
	Let $G$ be a group and $H \leq G$. The commutator subgroup $G'$ is the smallest normal subgroup such that $G/G'$ is abelian. Therefore, if $G' \cap H = \{1\}$ and 	
	\begin{align*}
	\max\{|O_p(G/G'):O_p(HG'/G')| \mid p \text{ prime divisor of } G/G' \} > \rk(G/G'),
	\end{align*}
	we can apply Corollary \ref{CorAbelianGrp}. This yields that there exists a $G$-invariant generating transversal for $H$ in $G$ containing 1. 
\end{Rem}

\newpage
\section{Conjecture for RCC loop folders}
\lhead{\slshape 6 \quad  Conjecture for RCC loop folders}

Let $G$ be a finite group and let $H$ be a subgroup of $G$.
In Lemma \ref{GStrichSchnitH} we showed that if $G' \cap H = \{1\}$, then there exists a $G$-invariant transversal containing 1. Now we want to examine under which conditions the converse statement is true.

Note that the assumptions in Lemma \ref{GStrichSchnitH} imply that $H$ is abelian. 
The next example shows that the existence of a $G$-invariant transversal does not imply that $H$ is abelian.

\begin{Example}
Let $D_{12}:=\langle s,t \mid s^6=t^2=1, tst=s^{-1} \rangle$ be the dihedral group of order 12 and let $H:=\langle s^2,t \rangle \leq D_{12}$. An easy calculation shows that $Z(D_{12})=\langle s^3 \rangle$ and that $D_{12}'=\langle s^2 \rangle$. Hence, $T:=\{1, s^3\}$ is a $D_{12}$-invariant transversal for $H \backslash D_{12}$ but $H$ is not abelian and $D'_{12} \cap H = D_{12}'$.
\end{Example}

Thus, the condition that $H$ is abelian is necessary and we formulate the converse statement as following conjecture:

\begin{Conj}
	\label{Conj}
	If $H$ is abelian and there exists a $G$-invariant transversal $T$ for $H \backslash G$ with $1 \in T$, i.e. $(G,H,T)$ is an RCC loop folder, then we have $G' \cap H = \{1\}$.
\end{Conj}

We verified this conjecture for all non-abelian groups of order smaller than 40 and for all envelops of RCC loops of order smaller than 30 (see \cite{articdiss}) with the computational algebra system GAP \cite{gap}.

One possibility to prove Conjecture \ref{Conj} is to show that $H$ posses a normal complement. Suppose given a normal complement $N$ of $H$. Then $G/N \cong H$ is abelian and thus, $G' \leq N$. This yields that $H \cap G' \leq H \cap N = \{1\}$.

This is the case for example in Chapter 3, where $G$ is a Frobenius group with an abelian complement containing $H$.

Both Burnside and Zappa stated interesting conditions for H possessing a normal complement.

\begin{Theorem}[Burnside,{\cite[IV, Hauptsatz 2.6]{huppert}}]
	Suppose that $H$ is a Sylow subgroup of $G$ such that $H \leq Z(N_G(H))$. Then there exists a normal subgroup $N$ such that $G/N \cong H$.	
\end{Theorem}

Lemma \ref{NormalizerCentralizer} shows that the Theorem of Burnside is applicable in our case.

\pagebreak

\begin{Cor}
	If $(G,H,T)$ is an RCC loop folder and $H$ is an abelian Sylow subgroup of $G$, then by Lemma \ref{NormalizerCentralizer} we have $C_G(H)=N_G(H)$. Thus, Burnside's theorem yields that $G' \cap H = \{1\}$. In particular, Conjecture \ref{Conj} holds if $H$ is an abelian Sylow subgroup.
\end{Cor}

In the next section, we will employ a theorem of Zappa.

\subsection{Existence of an $H$-invariant transversal}

In this section we suppose that there exists a transversal of $H \backslash G$, which is $H$-invariant.
In the literature, an $H$-invariant transversal of $H\backslash G$ is also called a distinguished system of coset representatives.
We show that our conjecture is true if $H$ is additionally a Hall subgroup and we give a criterion for the existence of a distinguished system of coset representatives for $H \backslash G$.

Zappa showed that if $H$ is a nilpotent Hall subgroup of $G$ possessing a distinguished system of coset representatives, then $G$ contains a normal subgroup $N$ such that $G=HN, H \cap N = \{1\}$ (see \cite[Proposizione XIV 12.1]{zappabook}). 
In \cite{kochendorffer}, Kochendörffer generalised this statement.
Here we prove a slightly weaker theorem than Zappa's theorem, relying on the idea of the proof of Kochendörffer.

\begin{Theorem}
	Let $G$ be a finite group and let $H$ be an abelian Hall-subgroup of $G$. Suppose that there exists an $H$-invariant transversal $T$ for $H \backslash G$.  Then we have $G' \cap H = \{1\}$.
\end{Theorem}

\begin{proof}
	The map 
	\begin{align*}
		\tau: G \rightarrow H , x \mapsto \prod_{t \in T} \lambda_x^T(t),
	\end{align*}
	where $\lambda_x^T(t)$ is the unique element $h\in H$ such that $tx=ht'$ for some $t' \in T$, is called transfer from $G$ to $H$ and it is a group homomorphism (see \cite[IV, Hauptsatz 1.4]{huppert}).
	
	Let $h \in H$ and let $h':=\lambda_h^T(t)$ for some $t \in T$. Then we have $th=h't'$ for some $t' \in T$. 
	It follows that $h{h'}^{-1}=t^{-1}h't'{h'}^{-1} \in H$ and since $T$ is $H$-invariant, we have $h't'{h'}^{-1} \in T$. 
	Thus, $t=h't'{h'}^{-1}$. This yields that $h=h'=\lambda_h^T(t)$ and hence, we obtain
	\begin{align*}
		\tau(h)=\prod_{t \in T}\lambda_h^T(t)=\prod_{t \in T}h=h^{|G:H|}.
	\end{align*}
	Now we show that the map $f: H \rightarrow H, h \mapsto h^{|G:H|}$ is bijective, for which it suffices to show that $f$ is injective.
	Let $h \in \ker f$. As $h^{|G:H|}=1$, it follows that $|h|$ divides $|G:H|$. But since $H$ is a Hall-subgroup, $|H|$ and $|G:H|$ are coprime. Thus, $h=1$ and the map $f$ is bijective. In particular, the restriction of $\tau$ to $H$ is bijective. This yields that $\ker \tau \cap H = \{1\}$. Furthermore, we have that $\bigslant{G}{\ker \tau}$ is abelian, because the image of $\tau$ is abelian as subgroup of $H$. Hence, $G' \leq \ker \tau$. We conclude that $G' \cap H \leq \ker \tau \cap H = \{1\}$.
\end{proof}

In \cite{felsch}, Felsch gives existence criteria for a distinguished system of coset representatives for $H \backslash G$ and from \cite{felsch} we obtain some more properties of groups possessing a distinguished system of coset representatives. We show these criteria in a more general approach. Therefore, we give a definition, first.

\begin{Def}[{\cite[1.1 Definition]{felsch}}]
	Let  $U \leq H \leq G $ be subgroups. \linebreak
	A transversal $S$ of $H\backslash G$ is called a $(U,H,G)$-system if $S$ is $U$-invariant.
\end{Def}
Note that an $(H,H,G)$-system is an $H$-invariant transversal for $H\backslash G$ or a distinguished system of coset representatives for $H\backslash G$.

\begin{Theorem}[{\cite[2.1 Satz]{felsch}}]
	\label{Felsch}
	Let $U \leq H \leq G$ be subgroups and let $S$ be a transversal for the double cosets of $H$ and $U$ in $G$. Then the following statements are equivalent:
	\begin{enumerate}[a)]
		\item There exists a $(U,H,G)$-system.
		\item $g \in HC_G(U \cap H^g)$ for all $g \in G$.
		\item $s \in HC_G(U \cap H^s)$ for all $s \in S$.
		\item For every $g \in G$ there exists $h \in H$, such that $x^h=x^g$ for all $x \in U^{g^{-1}}\cap H$.
		\item For every $s \in S$ there exists $h \in H$, such that $x^h=x^s$ for all $x \in U^{s^{-1}}\cap H$.
	\end{enumerate}
\end{Theorem}

To prove this theorem, we first prove two lemmas. 

\begin{Lemma}[{\cite[2.2 Hilfssatz]{felsch}}]
	\label{HelpFelsch1}
	Let $U \leq H \leq G$  be subgroups. Let $R$ be a $(U,H,G)$-system and $r \in R$. Then $r \in C_G(U \cap H^r)$.
\end{Lemma}
\begin{proof}
	Let $x \in U \cap H^r$. As $x \in H$, we have 
	\begin{align*}
		x^{-1}rx \in x^{-1}rH^r=x^{-1}Hr=Hr.
	\end{align*}
	 On the other hand, we have $x^{-1}rx \in x^{-1}Rx=R$. Thus, it follows that  $x^{-1}rx=r$ and since $x$ was chosen arbitrary, we conclude $r \in C_G(U \cap H^r)$.
\end{proof}

\begin{Lemma}[{\cite[2.3 Hilfssatz]{felsch}}]
	\label{HelpFelsch2}
	Let $U \leq H \leq G$ be subgroups and let $S$ be a transversal for the double cosets of $H$ and $U$ in $G$. If for every $s \in S$ there exists an element $s^* \in HsU$ such that $s^* \in C_G(H \cap H^{s^*})$, then the set of the elements $s^*$ and their conjugates under $U$ is a $(U,H,G)$-system.
\end{Lemma}
The proof is based on the proof of \cite[Lemma 2]{zappa2}.
\begin{proof}
	Set $S^*:=\{s^* \mid s \in S\}$ and
	$T:=\displaystyle\bigcup_{u \in U}\bigcup_{s \in S^*} u^{-1}su$.
	Clearly, $S^*$ is a transversal for the double cosets of $H$ and $U$ in $G$ and hence, it follows that 
	\begin{align*}
	\bigcup_{t \in T} Ht =	\bigcup_{u \in U}\bigcup_{s \in S^*} H su = \bigcup_{s \in S^*} HsU=G.
	\end{align*}
	Thus, it remains to show that at most one element of each coset of $H$ in $G$ lies in $T$. Suppose that $Ht_1=Ht_2$ for some $t_1,t_2 \in T$. Then $t_1=ht_2$ for some $h \in H$ and we have $t_1=x^{-1}s_1x$ and $t_2=y^{-1}s_2y$ for some $x,y \in U$ and $s_1,s_2 \in S^*$. As $Ht_1U=Ht_2U$, we obtain that $Hs_1U=Hs_2U$ and hence,  $s_1=s_2$. It follows from $x^{-1}s_1x=hy^{-1}s_1y$ and $xhy^{-1} \in H$ that 
	\begin{align*}
		xy^{-1}=s_1^{-1}xhy^{-1}s_1 \in H^{s_1}.
	\end{align*} 
	Furthermore, we have $xy^{-1} \in U$ and hence, $xy^{-1} \in U \cap H^{s_1}$. The assumption $s_1 \in C_G(U \cap H^{s_1})$ yields that 
	\begin{align*}
		xy^{-1}=s_1xy^{-1}s_1^{-1}=xhy^{-1}
	\end{align*}
 	and thus, $t_1=t_2$. In conclusion, $T$ is a $(U,H,G)$-system.
\end{proof}

Now we can prove Theorem \ref{Felsch}.

\begin{proof}[Proof of Theorem \ref{Felsch}]
\textbf{a) $\Rightarrow$ b): }
Let $R$ be a $(U,H,G)$-system and let $g \in G$. 
Then $g=hr$ for some $h\in H$ and $r \in R$. Lemma \ref{HelpFelsch1} yields that $r \in C_G(U \cap H^r)$. Since we have $r^{-1}Hr=g^{-1}Hg$, it follows that $r \in C_G(U \cap H^r)=C_G(U \cap H^g)$ and hence, we have $g \in HC_G(U \cap H^g)$.

\textbf{b) $\Rightarrow$ d):} Let $g \in G$. By our assumption, we have $g=hr$ for some $h \in H$ and $r \in C_G(U \cap H^g)$. Let $x \in U^{g^{-1}} \cap H$. Then $x^g \in U \cap H^g$. It follows that $x^h=x^{gr^{-1}}=x^g$. Since $g$ and $x$ were chosen arbitrarily, statement d) follows.

\textbf{d) $\Rightarrow$ e)} Statement e) follows directly, since $S \subseteq G$.

\textbf{e) $\Rightarrow$ c)} Let $s \in S$. By assumption, there exists $h \in H$ such that $x^h=x^s$ for all $x \in U^{s^{-1}} \cap H$. This implies that $(y^{s^{-1}})^h=(y^{s^{-1}})^s=y$ for all $y \in U \cap H^s$ and hence, we have $c:=s^{-1}h \in C_G(U \cap H^s)$. This yields that
\begin{align*}
s=hc^{-1} \in HC_G(U \cap H^s).
\end{align*}

\textbf{c) $\Rightarrow$ a)} Let $s \in S$. Then it follows from statement c) that $s=hs^*$ for some $h \in H$ and $s^* \in C_G(U \cap H^s)$. Since $H^s=H^{s^*}$, we have $s^* \in C_G(U \cap H^{s^*})$ and we also obtain that $s^*=h^{-1}s \in HsU$. Thus, the conditions of Lemma \ref{HelpFelsch2} are fulfilled and it follows that there exists a $(U,H,G)$-system.                                              
\end{proof}

With Theorem \ref{Felsch} we obtain directly the following criterion for the existence of an $H$-invariant transversal for $H \backslash G$ if $H$ is abelian.

\begin{Cor}[{\cite[2.5 Korollar]{felsch}}]
 Let $H$ be abelian. Then $H$ possess a distinguished system of coset representatives if and only if $g \in C_G(H \cap H^g)$ for all $g \in G$.
\end{Cor}

The proofs of the previous theorem and the previous lemmas rely on the fact that $U \leq H$. 
Thus, the stronger condition of the existence of a $G$-invariant transversal for $H \backslash G$ has to be investigated with different methods.

\subsection{Existence of a $G$-invariant transversal}

In the next example we show that it is not enough to demand that there exists a $H$-invariant transversal for $H \backslash G$ in general. For the Conjecture \ref{Conj}, we need $G$-invariance, as the following example shows.
	
	\begin{Example}
		Let $G:=Q_8$ and let $H:=Z(G)$. Then $H$ is abelian and every transversal of $H \backslash G$ is $H$-invariant, but there does not exist any $G$-invariant transversal for $H \backslash G$. Furthermore,  we have $H=Z(G)=G'$. Thus, $H \cap G'\neq \{1\} $. 
	\end{Example}

	However, if $G$ is a group of order $p^3$ with $p \in \Prim$ and there exists a $G$-invariant transversal for $H\backslash G$, then the Conjecture \ref{Conj} holds.

	\begin{Lemma}
		\label{groupp3}
		Let $G$ be a non-abelian $p$-group of order $p^3$ with $p \in \Prim$. Then we have $G'=Z(G)$ and $|Z(G)|=|G'|=p$.
	\end{Lemma}
	\begin{proof}
		Because $G$ is a non-trivial $p$-group, the center $Z(G)$ is non-trivial. As $G$ is non-abelian, it follows that $Z(G) \neq G$. Moreover, $G/Z(G)$ is not cyclic and therefore, we obtain $|G:Z(G)|=p^2$ and $|Z(G)|=p$.
		Since every group of order $p^2$ is abelian, we have $G' \leq Z(G)$.
		As $G$ is non-abelian, $G'$ is non-trivial and thus, $G'=Z(G)$.
	\end{proof}
	
	\pagebreak
	
	\begin{Lemma}
		Let $G$ be a $p$-group and $H \leq G$ abelian. Suppose that there exists a $G$-invariant transversal $T$ of $H \backslash G$ containing 1. If we have $G'=Z(G)$ and $|Z(G)|=|G'|=p$, then we have $G' \cap H = \{1\}$.
	\end{Lemma}
	\begin{proof}
		Since $T$ is $G$-invariant, $T$ is a union of conjugacy classes of $G$. As $1 \in T$, we know that $T$ contains at least $p$ conjugacy classes with exactly one element. Therefore, $T$ contains at least $p$ elements of $Z(G)$ and hence, $G'=Z(G) \subseteq T$. It follows that $G' \cap H \subseteq T \cap H = \{1\}$.
	\end{proof}
	
	This Lemma and Lemma \ref{groupp3} yield that the Conjecture \ref{Conj} holds for groups of order $p^3$ with $p \in \Prim$.

\newpage
\lhead{}
\bibliographystyle{abbrv}
\bibliography{projekt}

\end{document}